\documentclass{birkjour}
\usepackage{amsmath,amsfonts,amssymb}%

\usepackage[all]{xy}
\usepackage[usenames]{color}
\usepackage{setspace}
\usepackage{url}
\usepackage{halloweenmath}
\usepackage{amsmath, amsthm, amsfonts, amssymb, stmaryrd}%
\usepackage{mathrsfs}
\usepackage{tikz}
\usetikzlibrary{decorations.pathmorphing}
\usepackage{bm}
\usepackage{enumerate}
\usepackage{microtype}
\usepackage{xparse, expl3}
\usepackage{aliascnt}

\usepackage{hyperref}
\hypersetup{
    colorlinks = true,
    citecolor=darkblue,
    linkcolor=darkblue,
    urlcolor=darkblue
}
\usepackage[capitalize]{cleveref}
\crefname{theoremvArIAblE}{Theorem}{Theorems}
\crefname{propositionvArIAblE}{Proposition}{Propositions}
\crefname{lemmavArIAblE}{Lemma}{Lemmas}
\crefname{corollaryvArIAblE}{Corollary}{Corollaries}
\crefname{claimvArIAblE}{Claim}{Claims}
\crefname{definitionvArIAblE}{Definition}{Definitions}
\crefname{examplevArIAblE}{Example}{Examples}
\crefname{remarkvArIAblE}{Remark}{Remarks}
\crefname{conjecturevArIAblE}{Conjecture}{Conjectures}
\crefname{equation}{Equation}{Equations}
\crefname{section}{Section}{Sections}
\crefname{subsection}{Section}{Sections}

\definecolor{LightGray}{rgb}{.6,.6,.6}
\definecolor{darkblue}{rgb}{0.0,0.0,0.6} 

\def\xx{{\EM{\bm{x}}}}

\def\zz{{\EM{\bm{z}}}}

\def\Aut{{\mathrm{{Aut}}}}

\def\Lempty{\EM{\mc{L}}}

\def\Lomega#1{\EM{\Lempty_{#1, \w}}}

\def\Lkappalambda#1#2{{\EM{\Lempty_{#1, #2}}}}

\def\Lkw{\Lkappalambda{\kappa}{\w}}

\def\Lww{\Lomega{\w}}

\def\Lwow{\Lomega{\w_1}}
\def\Lkpw{\Lomega{\kappa^+}}
\def\Liw{\Lomega{\infty}}

\newcommand\rest[1][]{\EM{\!\restriction_{#1}}}

\newcommand\pars{\,\cdot\,}

\newcommand{\defn}[1]{{\textbf{#1}}}

\def\ZFC{\text{ZFC}}

\def\Powerset{\mathcal{P}}
\def\PowersetFin{\Powerset_{<\w}}

\newcommand{\Caratheodory}{Carath\'{e}odory}

\DeclareMathOperator{\dom}{dom}

\DeclareMathOperator{\rank}{rank}

\DeclareMathOperator{\CARD}{CARD}

\DeclareMathOperator{\Rel}{Rel}
\DeclareMathOperator{\arity}{ar}

\DeclareDocumentCommand{\forces}{d[]}
{
\IfNoValueTF{#1}
	{
	\EM{\Vdash}
	}
	{
	\EM{\Vdash_{#1}}
	}
}

\DeclareDocumentCommand{\Generic}{d[]}
{
\IfNoValueTF{#1}
	{
	\mathbb{G}
	}
	{
	\mathbb{#1}
	}
}

\DeclareDocumentCommand{\GenericUnion}{d[]}
{
\IfNoValueTF{#1}
	{
	\mathfrak{G}
	}
	{
	\mathfrac{#1}
	}
}

\DeclareDocumentCommand{\RRF}{d<> d()}
{
\IfNoValueTF{#1}
{
    \IfNoValueTF{#2}
    {
    \mathcal{R}
    }
    {
    \mathcal{R}(#2)
    }
}
{
    \IfNoValueTF{#2}
    {
    \mathcal{R}_{#1}
    }
    {
    \mathcal{R}_{#1}(#2)
    }
}
}

\DeclareDocumentCommand{\Cantor}{d[]}
{
\IfNoValueTF{#1}
	{
	2^\w
	}
	{
	2^{#1}
	}
}

\DeclareDocumentCommand{\Lebesgue}{d<> d[]}
{
\IfNoValueTF{#1}
    {
    \IfNoValueTF{#2}
        {
        \lambda
        }
        {
        \lambda_{#2}
        }
    }
    {
    \IfNoValueTF{#2}
        {
        \lambda^{(#1)}
        }
        {
        \lambda^{(#1)}_{#2}
        }
    }
}

\DeclareDocumentCommand{\NonRed}{d()}
{
\IfNoValueTF{#1}
	{
	\mathfrak{NR}
	}
	{
	\mathfrak{NR}(#1)
	}
}

\DeclareDocumentCommand{\Str}{d<> d[] d()}
{
\IfNoValueTF{#1}
{
\IfNoValueTF{#2}
	{
	\EM{\StrMathOp(#3)}
	}
	{
	\EM{\StrMathOp_{#2}(#3)}
	}
}
{
\IfNoValueTF{#2}
	{
	\EM{\StrMathOp^{#1}(#3)}
	}
	{
	\EM{\StrMathOp^{#1}_{#2}(#3)}
	}
}
}

\newcommand{\skullsymb}{\boxcircle}

\DeclareDocumentCommand{\Skeleton}{d<> d[] d()}
{
\IfNoValueTF{#1}
{
        \IfNoValueTF{#2}
        	{
        	\skullsymb_{#3}
        	}
        	{
        	\skullsymb_{#3:#2}
        	}
}
{
\IfNoValueTF{#2}
        {
        \skullsymb_{#3:#1}
        }
        {
        \skullsymb_{#3:#1, #2}
        }
}
}

\DeclareDocumentCommand{\Proj}{s d<> d[] d()}
{
\IfBooleanTF{#1}
{
\IfNoValueTF{#2}
    {
    \IfNoValueTF{#3}
	{
        \IfNoValueTF{#4}
    	{
            \EM{ERROR}
            }
    	{
            \EM{\Pushforward[\pi](#4)}
            }
        }
        {
        \IfNoValueTF{#4}
    	{
            \EM{\Pushforward[\pi_{#3}]}
            }
    	{
            \EM{\Pushforward[\pi_{#3}](#4)}
            }
        }
    }
    {
    \IfNoValueTF{#3}
	{
        \IfNoValueTF{#4}
    	{
            \EM{\Pushforward\pi_{#2}]}
            }
    	{
            \EM{\Pushforward[\pi_{#2}](#4)}
            }
        }
        {
        \IfNoValueTF{#4}
    	{
            \EM{\Pushforward[\pi_{#2:#3}]}
            }
    	{
            \EM{\Pushforward[\pi_{#2:#3}](#4)}
            }
        }
    }
}
{
\IfNoValueTF{#2}
    {
    \IfNoValueTF{#3}
	{
        \IfNoValueTF{#4}
    	{
            \EM{ERROR}
            }
    	{
            \EM{\pi(#4)}
            }
        }
        {
        \IfNoValueTF{#4}
    	{
            \EM{\pi_{#3}}
            }
    	{
            \EM{\pi_{#3}(#4)}
            }
        }
    }
    {
    \IfNoValueTF{#3}
	{
        \IfNoValueTF{#4}
    	{
            \EM{\pi_{#2}}
            }
    	{
            \EM{\pi_{#2}(#4)}
            }
        }
        {
        \IfNoValueTF{#4}
    	{
            \EM{\pi_{#2:#3}}
            }
    	{
            \EM{\pi_{#2:#3}(#4)}
            }
        }
    }
}
}

\newcommand{\llrr}[1]{{\llbracket #1 \rrbracket}}

\DeclareDocumentCommand{\extent}{d<> d[] m}
{
\IfNoValueTF{#1}
{
\IfNoValueTF{#2}
	{
	\EM{\llrr{#3}}
	}
	{
	\EM{\llrr{#3}_{#2}}
	}
}
{
\IfNoValueTF{#2}
	{
	\EM{\llrr{#3}^{#1}}
	}
	{
	\EM{\llrr{#3}_{#2}^{#1}}
	}
}
}

\DeclareMathOperator{\ThOp}{Th}
\DeclareDocumentCommand{\Th}{d[]}
{
\IfNoValueTF{#1}
	{
	\ThOp
	}
	{
	\Th_{#1}
	}
}

\DeclareMathOperator{\StrMathOp}{Str}

\DeclareDocumentCommand{\ImEquiv}{d[] d()}
{
\IfNoValueTF{#2}
	{
	\EM{[\cdot]_{#1}}
	}
	{
	\EM{[#2]_{#1}}
	}
}

\DeclareDocumentCommand{\dual}{d()}
{
\IfNoValueTF{#1}
	{
	\EM{\hat{\ }}
	}
	{
	\EM{\hat{#1}}
	}
}

\DeclareDocumentCommand{\SizeAS}{d()}
{
\IfNoValueTF{#1}
	{
	\EM{|\cdot|}
	}
	{
	\EM{|#1|}
	}
}

\DeclareDocumentCommand{\compcK}{d[]}
{
\IfNoValueTF{#1}
	{
	\EM{\mathbb{K}}
	}
	{
	\EM{\mathbb{K}[#1]}
	}
}

\DeclareDocumentCommand{\AgeK}{d[]}
{
\IfNoValueTF{#1}
	{
	\bm{K}
	}
	{
	\bm{K}[#1]
	}
}

\DeclareDocumentCommand{\Rel}{d[]}
{
\IfNoValueTF{#1}
	{
	\EM{\mathcal{R}}
	}
	{
	\EM{\mathcal{R}_{#1}}
	}
}

\DeclareDocumentCommand{\Func}{d[]}
{
\IfNoValueTF{#1}
	{
	\EM{\mathcal{F}}
	}
	{
	\EM{\mathcal{F}_{#1}}
	}
}

\DeclareDocumentCommand{\ar}{d[]}
{
\IfNoValueTF{#1}
	{
	\EM{\arity}
	}
	{
	\EM{\arity_{#1}}
	}
}

\DeclareDocumentCommand{\Fn}{d<> d[] d()}
{
\IfNoValueTF{#3}
	{
	\EM{\textrm{Fn}(#1, #2)}
	}
	{
	\EM{\textrm{Fn}(#1, #2, #3)}
	}
}

\DeclareMathOperator{\Perm}{Sym}
\DeclareDocumentCommand{\Sym}{d()}
{
\EM{\Perm(#1)}
}

\DeclareDocumentCommand{\SymFin}{d()}
{
\EM{\Perm^{<\w}(#1)}
}

\DeclareDocumentCommand{\name}{d[] m}
{
\IfNoValueTF{#1}
	{
        \EM{\dot{#2}}
        }
	{
        \EM{\dot{#2}^{#1}}
	}
}

\DeclareDocumentCommand{\eqcl}{d[] d()}
{
\IfNoValueTF{#1}
	{
        \EM{[#2]}
        }
	{
        \EM{[#2]_{#1}}
	}
}

\DeclareDocumentCommand{\SymDiff}{}{\triangle}

\DeclareDocumentCommand{\CantorFiniteSeq}{}{\mathbb{C}}

\DeclareDocumentCommand{\Cohen}{d[] d()}
{
\mathfrak{C}_{#1}(#2)
}

\newcommand{\Fraisse}{\textrm{Fra\"iss\'e}}

\DeclareDocumentCommand{\tdcl}{d[] d()}
{
\IfNoValueTF{#1}
{
    \EM{\textbf{term}(#2)}
}
{
    \EM{\textbf{term}_{#1}(#2)}
}
}

\DeclareDocumentCommand{\gdcl}{d[] d()}
{
\IfNoValueTF{#1}
{
    \EM{\textrm{gcl}(#2)}
}
{
    \EM{\textrm{gcl}_{#1}(#2)}
}
}

\DeclareDocumentCommand{\qftp}{d[] d()}
{
\IfNoValueTF{#1}
{
    \EM{\quantfreetp(#2)}
}
{
    \EM{\quantfreetp_{#1}(#2)}
}
}

\DeclareMathOperator{\quantfreetp}{qtp}

\DeclareDocumentCommand{\Closure}{d<> d()}
{
\IfNoValueTF{#1}
{
    \EM{\textrm{cl}(#2)}
}
{
    \EM{\textrm{cl}_{#1}(#2)}
}
}

\DeclareDocumentCommand{\Pushforward}{d[] d()}
{
#1^*(#2)
}

\DeclareDocumentCommand{\BorelSets}{d<> d[] d()}
{
\IfNoValueTF{#1}
{
    \sigma(\Str[#2](#3))
}
{
    \sigma_{#1}(\Str[#2](#3))
}
}

\DeclareDocumentCommand{\BorelCodes}{d<> d[] d()}
{
\IfNoValueTF{#1}
{
    \mbb{B}_{#2:#3}
}
{
    \mbb{B}^{#1}_{#2:#3}
}
}

\DeclareDocumentCommand{\BasicBC}{d<> d[] d()}
{
\IfNoValueTF{#1}      
{
    \mbb{B}^{\mathfrak{b}}_{#2:#3}
}
{
    \mbb{B}^{{#1}, \mathfrak{b}}_{#2:#3}
}
}

\DeclareDocumentCommand{\AlmostBasicBC}{d<> d[] d()}
{
\IfNoValueTF{#1}      
{
    \mbb{B}^{\mathfrak{ab}}_{#2:#3}
}
{
    \mbb{B}^{{#1},\mathfrak{ab}}_{#2:#3}
}
}

\DeclareDocumentCommand{\BCspace}{d[] d()}
{
\IfNoValueTF{#1}
{
    \mathfrak{BS}(#2)
}
{
    \mathfrak{BS}_{#1}(#2)
}
}

\DeclareDocumentCommand{\BorelProj}{d[] d()}
{
\<:#2:\>_{#1}
}

\newcommand\widetildelong[1]{\mathord{\begin{tikzpicture}[baseline= (a.base)]
\node[inner sep=0pt] (a) {$#1$};
\draw[semithick,decorate,decoration={snake,amplitude=.3mm,segment length=0.8mm}] ([xshift=.15ex,yshift=0.5ex]a.north west) -- ([xshift=-.15ex,    yshift=0.5ex]a.north east);
\end{tikzpicture}}}

\DeclareDocumentCommand{\BCvalue}{d()}
{
\widetildelong{#1}
}

\makeatletter
\newsavebox{\@brx}
\newcommand{\lllangle}[1][]{\savebox{\@brx}{\(\m@th{#1\langle}\)}%
  \mathopen{\copy\@brx\kern-0.5\wd\@brx\usebox{\@brx}}}
\newcommand{\rrrangle}[1][]{\savebox{\@brx}{\(\m@th{#1\rangle}\)}%
  \mathclose{\copy\@brx\kern-0.5\wd\@brx\usebox{\@brx}}}
\makeatother

\DeclareDocumentCommand{\BCformula}{d()}
{
\lllangle #1 \rrrangle
}

\DeclareDocumentCommand{\CohenImage}{d[] d()}
{
\llparenthesis {#2} \rrparenthesis_{#1}
}

\DeclareDocumentCommand{\RandPO}{d()}
{
\IfNoValueTF{#1}
{
    \EM{\cO^+}
}
{
    \EM{\cO^+(#1)}
}
}

\DeclareDocumentCommand{\NonRedTuple}{d[] d()}
{
#2^{(#1)}
}

\DeclareDocumentCommand{\ClosureMap}{d[] d()}
{
\EM{\textrm{clMap}(#1, #2)}
}

\DeclareDocumentCommand{\HP}{}
{%
    \textrm{(HP)}%
}

\DeclareDocumentCommand{\JEP}{}
{%
    \textrm{(JEP)}%
}

\DeclareDocumentCommand{\AP}{}
{%
    \textrm{(AP)}%
}

\DeclareDocumentCommand{\SAP}{}
{%
    \textrm{(SAP)}%
}

\makeatletter
\newcommand{\dotminus}{\mathbin{\text{\@dotminus}}}

\newcommand{\@dotminus}{%
  \ooalign{\hidewidth\raise1ex\hbox{.}\hidewidth\cr$\m@th-$\cr}%
}

\DeclareDocumentCommand{\LangOne}{}
{
    \Lang_1
}

\DeclareMathOperator{\Lang}{\mathscr{L}}

\DeclareDocumentCommand{\quadiff}{}
{
    \EM{\quad\text{ if and only if }\quad}
}

\DeclareDocumentCommand{\qquadiff}{}
{
    \qquad\text{if and only if}\qquad
}

\DeclareDocumentCommand{\floor}{m}
{
\lfloor #1\rfloor
}

\newcommand{\sC}{\mathscr{C}}
\newcommand{\sCle}{\leqslant}

\def\aa{{\EM{\bm{a}}}}
\def\bb{{\EM{\bm{b}}}}

\def\bb{{\EM{\bm{b}}}}

\def\cA{{\EM{\mathcal{A}}}}
\def\cB{{\EM{\mathcal{B}}}}
\def\cC{{\EM{\mathcal{C}}}}

\def\cE{{\EM{\mathcal{E}}}}

\def\cM{{\EM{\mathcal{M}}}}
\def\cN{{\EM{\mathcal{N}}}}
\def\cO{{\EM{\mathcal{O}}}}

\def\cV{{\EM{\mathcal{V}}}}

\def\cZ{{\EM{\mathcal{Z}}}}

\def\w{\EM{\omega}}

\def\Rationals{{\EM{{\mbb{Q}}}}}

\def\Reals{\EM{\mbb{R}}}

\def\^{\EM{{}^{\And}}}

\def\Or{\EM{\vee}}
\def\And{\EM{\wedge}}

\def\<{\EM{\langle}}
\def\>{\EM{\rangle}}

\def\indent{\hspace*{2em}}

\def\EM#1{\ensuremath{#1}}
\def\mbb#1{\EM{\mathbb{#1}}}
\def\mc#1{\EM{\mathcal{#1}}}

\def\st{\,:\,}
\def\:{\colon}

\providecommand{\dotdiv}{
  \mathbin{
    \vphantom{+}
    \text{
      \mathsurround=0pt 
      \ooalign{
        \noalign{\kern-.35ex}
        \hidewidth$\smash{\cdot}$\hidewidth\cr 
        \noalign{\kern.35ex}
        $-$\cr 
      }%
    }%
  }%
}

\DeclareDocumentCommand{\RightJustify}{m}{\hspace*{\fill}\mbox{#1}\penalty-9999\relax}

\newcommand{\bA}{\mbb{A}}
\newcommand{\bB}{\mbb{B}}
\newcommand{\bP}{\mbb{P}}
\newcommand{\bQ}{\mbb{Q}}
\DeclareMathOperator{\RO}{RO}

\newcommand{\forcingisom}{\approx}
\newcommand{\isom}{\cong}

\DeclareDocumentCommand{\DeclareCounter}{m}%
{%
 	\ifcsname c@#1\endcsname%
	\else%
		\newcounter{#1}%
	\fi%
}

\DeclareDocumentCommand{\MyQED}{}{\qed}

\DeclareDocumentEnvironment{proof}{d() D[]{\MyQED} d<>}
{
\noindent\IfNoValueTF{#1}
{\emph{Proof.\!\!}}
{\emph{Proof\ #1.\ }}
}
{
\IfValueTF{#3}
	{%
	\ExplSyntaxOff
	\RightJustify{\EM{#2_{#3}}}
	\vspace{1ex}
	}
	{
	\RightJustify{\EM{#2}}
	\vspace{1ex}
	}
}

\DeclareCounter{ProofLabelcOUntEr}
\DeclareDocumentCommand{\ProofLabel}{}{%
\addtocounter{ProofLabelcOUntEr}{1}
\label{cUrrEntProoflAbEl\arabic{ProofLabelcOUntEr}}
}

\DeclareDocumentCommand{\ProofRef}{D<>{1}}
{%
\ref{cUrrEntProoflAbEl\arabic{ProofcOUntEr#1}}
}
\DeclareDocumentCommand{\ProofCref}{D<>{1}}
{%
\cref{cUrrEntProoflAbEl\arabic{ProofcOUntEr#1}}
}

\DeclareDocumentEnvironment{Proof}{d() D[]{\MyQED} D<>{1}}
{
\DeclareCounter{ProofcOUntEr#3}%
\setcounter{ProofcOUntEr#3}{\value{ProofLabelcOUntEr}}%
\begin{proof}(#1)[#2]<\ProofRef<#3>>%
}
{
\end{proof}
}

\ExplSyntaxOn
\def\TheoremDepth{section}

\DeclareDocumentCommand{\DeclareTheorem}{m o m o}{%
\IfNoValueTF{#4}
	{%
	\IfNoValueTF{#2}
		{%
		\newtheorem{#1vArIAblE}{#3}
		}
		{%
		\newaliascnt{#1vArIAblE}{#2vArIAblE}
		\newtheorem{#1vArIAblE}[#1vArIAblE]{#3}
		\aliascntresetthe{#1vArIAblE}
		}
	}
	{%
	\newtheorem{#1vArIAblE}{#3}[#4]%
	}
\newtheorem*{#1vArIAblE*}{#3}

\DeclareDocumentEnvironment{#1}{o o}

	{
	\IfValueT{##2}%
		{
		\begin{spacing}{##2}
		}
	\IfValueTF{##1}
		{
		\begin{#1vArIAblE}[##1]
		}
		{
		\begin{#1vArIAblE}
		}
	\ProofLabel
	}
	{
	\IfValueT{##2}%
		{
		\end{spacing}{##2}
		}
	\end{#1vArIAblE}
	}

\DeclareDocumentEnvironment{#1*}{o o}

	{
	\IfValueT{##2}%
		{
		\begin{spacing}{##2}
		}
	\IfValueTF{##1}
		{
		\begin{#1vArIAblE*}[##1]
		}
		{
		\begin{#1vArIAblE*}
		}
	}
	{
	\IfValueT{##2}%
		{
		\end{spacing}{##2}
		}
	\end{#1vArIAblE*}
	}
}
\ExplSyntaxOff

\theoremstyle{plain}
\DeclareTheorem{theorem}{Theorem}[\TheoremDepth]
\DeclareTheorem{proposition}[theorem]{Proposition}
\DeclareTheorem{lemma}[theorem]{Lemma}
\DeclareTheorem{corollary}[theorem]{Corollary}
\DeclareTheorem{claim}[theorem]{Claim}

\theoremstyle{definition}
\DeclareTheorem{definition}[theorem]{Definition}
\DeclareTheorem{axiom}[theorem]{Axiom}
\DeclareTheorem{convention}[theorem]{Convention}

\theoremstyle{remark}
\DeclareTheorem{remark}[theorem]{Remark}
\DeclareTheorem{example}[theorem]{Example}
\DeclareTheorem{question}[theorem]{Question}
\DeclareTheorem{conjecture}[theorem]{Conjecture}

\begin{document}

\title{Forcing with Invariant Measures}

\thanks{The third author's research has been supported by a grant from IPM (No.~1402030417).}

\begin{abstract}
This paper introduces 
a model-theoretic generalization of the notion of forcing with random reals, 
in which forcing gives rise to \emph{random generic structures}. 
Specifically, we consider forcing with $\kappa$-Borel probability measures on the space of $\Lang$-structures with a 
(possibly uncountable) infinite set $X$, focusing on those that are invariant under the action of the symmetric group $\Sym(X)$.
We demonstrate how any 
$\Sym(X)$-invariant measure where $X$ is countable
can be uniquely extended to a
$\Sym(Y)$-invariant measure where $Y$ is uncountable,
and
prove that forcing with such measures
satisfies the countable chain condition.
We also show that we can uniformly distinguish between 
these 
random generic structures
and the \emph{Cohen generic structures} that arise from forcing with a strong \Fraisse\ class: There is a $\kappa$-Borel set of low complexity that contains every Cohen generic structure that is not highly homogeneous but contains no random generic structure, implying that a structure that is not highly homogeneous cannot be both Cohen generic and random generic.
Finally, 
we answer an open question of Kostana in the case of $\w_1$, by establishing a connection between forcing with a strong \Fraisse\ class and Cohen forcing.
\end{abstract}

\author[Ackerman]{Nathanael Ackerman}
\address{Harvard University\\
Cambridge, MA 02138\\
USA}
\email{nate@aleph0.net}

\author[Freer]{Cameron Freer}
\address{Department of Brain and Cognitive Sciences\\
Massachusetts Institute of Technology\\
Cambridge, MA 02139\\
USA}
\email{freer@mit.edu}

\author[Golshani]{Mohammad Golshani}
\address{School of Mathematics\\
Institute for Research in Fundamental Sciences (IPM)\\
P.O. Box: 19395-5746\\
Tehran\\
Iran}
\email{golshani.m@gmail.com}

\author[Mirabi]{Mostafa Mirabi}
\address{Department of Mathematics and Computer Science\\
The Taft School\\
Watertown, CT 06795\\
USA\\
and \\
Department of Mathematics and Computer Science\\
Wesleyan University\\
Middletown, CT 06459\\
USA}
\email{mmirabi@wesleyan.edu}

\author[Patel]{Rehana Patel}
\address{Wesleyan University\\
Middletown, CT 06459\\
USA}
\email{rpatel@wesleyan.edu}


\newcommand\blfootnote[1]{%
  \begingroup
  \renewcommand\thefootnote{}
  \footnote{#1}%
  \addtocounter{footnote}{-1}
  \endgroup
}

\blfootnote{The present paper is the winner of the Christine Ladd-Franklin Logic Prize for the USA 2025, part of the third edition of the World Logic Prizes Contest.}

\maketitle
\tableofcontents

\section{Introduction}
In 1963, 
Cohen \cite{MR157890} introduced the notion of \emph{forcing} to prove the independence of the continuum hypothesis and the axiom of choice from the axioms of Zermelo-Fraenkel set theory.  Forcing enables one to add a new, so-called \emph{generic}, object to an existing \emph{ground} model of set theory.
The method uses a partial order of \emph{conditions} from the ground model, each of which can be thought of as an approximation to a generic object.
In \emph{Cohen forcing}, the conditions are 
finite partial functions from $X$ to $\{0, 1\}$ for a set $X$, 
and the resulting generic function is a total function from $X$ to $\{0, 1\}$,
i.e., an element of the generalized Cantor space $\Cantor[X]$.
(In Cohen's specific method for the independence of the continuum hypothesis, 
$X = \kappa \times \w$ for a cardinal $\kappa$,
so that the generic function witnesses
that $2^{\w} \ge \kappa$.)

In the 1970s, 
Robinson developed the technique of model-theoretic forcing
(\cite{MR272613}, \cite{Robinson1975-ROBFIM}, \cite{MR389581}), which has
provided deep insights into model-theoretic algebra and other areas of model theory.
For a language $\Lang$, 
this technique allows one to build a generic countably infinite $\Lang$-structure 
using conditions that are 
finite structures in a \emph{strong \Fraisse\ class},
where the underlying set of each condition
is a subset of some common countable set.
Recent work of Golshani \cite{Golshani} and Kostana \cite{Cohen} considers 
forcing with strong \Fraisse\ classes 
where the underlying set of each condition
is a subset of some common \emph{uncountable} set.

Given a strong \Fraisse\ class with countably many isomorphism types,
forcing with
it (restricted to structures with underlying set a subset of $\w$)
produces a countable \emph{\Fraisse\ limit}, i.e., a structure which is universal for the class and is ``maximally symmetric''.
Strong \Fraisse\ classes appear throughout mathematics. Examples include the class of all finite graphs, whose countable \Fraisse\ limit is the Rado graph, and the class of all finite linear orders, whose countable \Fraisse\ limit is the 
set of rational numbers as a linear order.

Note that a function from a set $X$ to $\{0, 1\}$ may be viewed as an \linebreak $\Lang_1$-structure with underlying set $X$, where $\LangOne = \{U\}$ for a unary relation symbol $U$.
Then Cohen forcing over finite partial functions can be viewed as
the special case of model-theoretic forcing 
where the conditions are those elements
of
the strong \Fraisse\ class of all finite $\Lang_1$-structures 
whose underlying set is a 
subset of $X$.
For this reason, we also 
use the term \emph{Cohen forcing} to refer to
forcing with a strong \Fraisse\ class in an arbitrary relational language,
and we call the generic object it produces a \emph{Cohen generic structure}.

Each comeager Borel subset of $\Cantor[\w]$ can be viewed as 
a property that holds of a ``typical'' element of $\Cantor[\w]$.
An important characterization of Cohen generic functions from $\w$ to $\{0, 1\}$ is that they are precisely those elements 
of $\Cantor[\w]$ 
that 
are in every comeager
Borel set 
that can be described in the ground model.
An alternative view is to consider
each Borel subset of $\Cantor[\w]$ of Lebesgue measure~1 to be
a property that holds of a ``typical'' element of $\Cantor[\w]$.
This exemplifies
a more general analogy between category and measure in the study of Cantor space, in which category is concerned with all \emph{possible} configurations, and measure is concerned with configurations that occur \emph{often}.

Under this analogy, the measure-theoretic counterpart
to Cohen forcing is
\emph{random forcing}, whose generic functions in $\Cantor[\w]$ are precisely those that 
are in every
Borel set  of Lebesgue measure~$1$ 
that can be described in the ground model.
This notion of forcing was introduced by Solovay in 1964 to show that
there is a model of set theory 
in which all subsets of the reals are Lebesgue-measurable \cite{MR265151}. 

Again viewing functions in $\Cantor[\w]$ as $\Lang_1$-structures, random forcing can be seen as building an $\Lang_1$-structure in a Lebesgue-random generic way. Specifically, 
a random generic is
isomorphic to
the $\Lang_1$-structure in which $U$ holds on an infinite-coinfinite subset of $\w$.
It is natural to ask how one might 
extend random forcing to a forcing notion that
produces a random generic structure in an arbitrary relational language.
To do so, one would need to identify appropriate analogues of Lebesgue measure with which to force.

A key property of Lebesgue probability measure on $2^\omega$ (equivalently, on the space of $\Lang_1$-structures on $\omega$) that makes it suitable for random forcing is that it is \emph{invariant} under permutations of the underlying set $\w$. This can be thought of as saying that the construction of the generic does not make use of any intrinsic relationship among the elements
of the underlying set (e.g., their ordering).
In this paper, we will propose 
that an appropriate extension of random forcing is
forcing with certain $\Sym(X)$-invariant probability measures on the space of $\Lang$-structures, where $X$ is an arbitrary set and $\Lang$ is an arbitrary relational language.

Invariant probability measures (on spaces of structures)
occur throughout mathematics.
They were studied in the 1930s by 
de~Finetti \cite{deFinetti} in the case of $\Lang_1$; these invariant measures are the distributions of \emph{exchangeable sequences}.
The de~Finetti representation theorem for invariant probability measures on $\Lang_1$-structures
plays a crucial role in the foundations of statistical inference and the philosophical debate between Bayesians and frequentists \cite{MR0440641}.

In the 1970s 
Aldous \cite{MR637937} and Hoover \cite{Hoover79} 
generalized
de Finetti's
representation theorem 
to invariant probability measures for relational languages of bounded arity.
In the 2000s,  work of Lov\'asz, Szegedy, and others \cite{MR3012035} 
rediscovered a version of this representation theorem in the context of graphs;
the objects representing invariant measures in 
this 
representation are called 
\emph{graphons},
and can be seen as limit objects of convergent sequences of large graphs.
This view of invariant measures provides a new perspective on deep theorems in combinatorics such as Szemer\'edi's regularity lemma \cite{MR369312}.

In the 1960s, Gaifman  \cite{MR175755}
studied the model theory of invariant probability measures for arbitrary relational languages, 
viewing them as a particular kind of ``probabilistic structure''.
The model theory of these probabilistic structures was further developed in recent years by Ackerman, Freer, and Patel. 
In particular, many classical results in infinitary logic were shown to have analogues for these probabilistic structures \cite{complete-classification}.

In this paper, we develop \emph{random forcing with invariant measures}, specifically random forcing using an invariant probability $\kappa$-measure 
on the space of $\Lang$-structures with a fixed underlying set, where $\Lang$ is an arbitrary relational language and $\kappa$ is an infinite ordinal.
We call the generic structures produced by these forcing notions \emph{random generic structures}.
We investigate the relationship between random generic structures and the invariant measures used to create them. 
This 
work can be viewed
as studying the connection between a ``random sample'' from a probabilistic structure and the probabilistic structure itself.
We also study the connections between random generic structures and Cohen generic structures.

Further, we
provide a robust tool for lifting measure-theoretic properties of invariant measures from Cantor space ($\Cantor[\w]$) to generalized Cantor space ($\Cantor[\kappa]$).
In addition, we show a way to generalize random forcing to produce uncountable as well as countable structures.
Our results provide a comprehensive theory of forcing using invariant measures, 
simultaneously generalizing random forcing with Lebesgue measure and model-theoretic forcing with strong \Fraisse\ classes.

\subsection{Outline of the Paper}

In \cref{Section: Preliminaries}, we review some of the ideas that will be needed for our work, including forcing, Boolean algebras, the logic action, $\kappa$-Borel codes, the almost-sure theory of a measure, invariant probability $\kappa$-measures, and \Fraisse\ limits. The notion of a $\kappa$-Borel code
is of special importance, as we will consider relativizations of measures to larger models of set theory, and $\kappa$-Borel codes give us an absolute way to describe them.

In \cref{Section: Examples on Cantor}, we review some of the key properties of
Cohen forcing over arbitrary sets as well as of random forcing with Lebesgue measure over arbitrary sets.
Because functions from a set $X$ to $\{0, 1\}$ are cryptomorphic to $\LangOne$-structures on $X$, 
the results in this section 
can be respectively viewed as a special case of Cohen forcing over arbitrary structures and of random forcing with respect to an appropriate invariant measure over arbitrary structures,
which
we study in \cref{Section: Forcing with Strong Fraisse Classes,Random Generic Structures Section}.

In \cref{Section: Forcing with Strong Fraisse Classes}, we review Cohen forcings over strong \Fraisse\ classes and consider the relationship between generics for
different strong \Fraisse\ classes. For an infinite set $X$, we show that forcing over $X$ using any non-trivial strong \Fraisse\ class with at most $|X|$-many isomorphism types will yield a model of set theory containing a Cohen generic function in $\Cantor[X]$. We also show that when $|X| \leq \w_1$ such forcings are equivalent to Cohen forcing over $X$, in the sense that they yield the same forcing extensions.  This answers an open question from \cite{Cohen} in the case of $\w_1$.

In \cref{Section: Invariant Measures},
we turn our focus to invariant probability $\kappa$-measures on the space of structures with a given infinite underlying set. 
In  \cref{subsec:liftings},
we show that every such invariant measure has a relativization to every model of set theory. We also show that for all infinite sets $X$ and $Y$, any $\Sym(Y)$-invariant measure
can be uniquely lifted to an analogous $\Sym(X)$-invariant measure.
In \cref{Section: Ergodic Extreme dissociated}, we 
consider three key properties that an invariant measure may satisfy, namely 
being ergodic, extreme, or dissociated. 
A fundamental result in probability theory is that these notions are equivalent for invariant probability measures on the space of structures with a countable underlying set.
We show that these three notions 
are each absolute between models of set theory.
We then use this fact along with our lifting technique from \cref{subsec:liftings} 
to show that the three notions coincide for invariant probability $\kappa$-measures on the space of structures with an arbitrary uncountable underlying set.

In \cref{Random Generic Structures Section}, we develop the notion of forcing with an invariant probability $\kappa$-measure $\mu$ to obtain a $\mu$-random generic structure. We show that forcing with such an invariant measure has the countable chain condition, and hence preserves all cardinals and cofinalities.
We also show that for every $\mu$-generic filter there is a unique $\mu$-random generic structure that characterizes the filter. This $\mu$-random generic structure can be thought of as a ``random sample'' from the invariant measure, and we study the relationship between such a sample and the measure it came from. Specifically, when $\mu$ is ergodic, the almost-sure theory of the measure is complete, and we show that every $\mu$-random generic structure
must satisfy the almost-sure theory of $\mu$.
This holds even when $\mu$ is \emph{properly ergodic}, in which case the almost-sure theory has no models in the ground model of set theory by \cite{AFKrP}.

In \cref{Section: Distinguishing Generic Structures}, we consider how much information can be recovered about an invariant probability $\kappa$-measure $\mu$ given access only to a random generic structure for $\mu$. We show that when $\mu$ is ergodic, one can completely recover $\mu$ from any random generic structure for $\mu$. 
Next we study the relationship between Cohen generic structures and random generic structures.
Certain \Fraisse\ classes have the property of being \emph{highly homogeneous}.
We show that for any infinite set $X$, there is a single $|X|$-Borel code such that for every Cohen generic structure for a strong \Fraisse\ class that is not highly homogeneous,
the Cohen generic structure
is contained in the set represented by the code. However, for any invariant probability $\kappa$-measure $\mu$, 
every 
random generic structure for $\mu$ is contained in the complement of the set represented by the code. As a consequence, 
if a structure is not highly homogeneous, it
cannot be simultaneously Cohen generic and random generic.

Finally, in \cref{Section: Conjectures}, we conclude with several natural conjectures.

\section{Preliminaries}  
\label{Section: Preliminaries}

\subsection{Notation}      

We work in an unnamed but fixed background model of \ZFC. We will use $X$, $Y$, $Z$ and their variants (e.g., $X_0$, $X_1$, or $X^*$) to represent sets in this background model of set theory, and $X$ and $Y$ will always be infinite. We write $\CARD$ to denote the collection of cardinals. For $n \in \w$, we define $[n]= \{0, \dots, n-1\}$. We write $\Cantor[Z]$ to denote the set of functions from $Z$ to $\{0, 1\}$. We write $\Powerset(Z)$ to denote the powerset of $Z$ (i.e., the collection of subsets of $Z$). We write $\PowersetFin(Z)$ to denote the collection of finite subsets of $Z$, and $\Powerset_{\w}(Z)$ to denote the collection of countably infinite subsets of $Z$. The symbol $\kappa$ will denote a fixed infinite ordinal,
although intuitively it may be helpful to think of $\kappa$ as a cardinal (but when we move to a larger model of set theory, we do not require $\kappa$ to stay a cardinal). 

We write $\<a_i\>_{i \in Z}$ to denote a sequence indexed by a set $Z$, and write
$\<a_i\>_{i \in Z} \subseteq B$ to mean that $a_i \in B$ for $i \in Z$.
For a sequence $\<r_i\>_{i \in Z} \subseteq [0, \infty)$, we define $\sum_{i \in Z} r_i = \sup \bigl\{\sum_{i \in Z_0} r_i \st Z_0 \in \PowersetFin(Z)\bigr\}$. For $r \in \Reals$, we let $\floor{r}$ be the largest integer with $\floor{r} \leq r$.
Given a finite tuple $\aa$ (of variables or elements of some set), we write $|\aa|$ to denote its length.

We write $Z_0 \SymDiff Z_1$ for the symmetric difference of $Z_0$ and $Z_1$, i.e., \linebreak $(Z_0 \setminus Z_1) \cup (Z_1 \setminus Z_0)$. 
When $\sim$ is an equivalence relation on $Z$ and $a \in Z$, we write $\eqcl[\sim](a)$ for the $\sim$-equivalence class of $a$. We let $\Sym(Z)$ be the group of permutations of $Z$, and for $n \in \w$ we write $\Sym(n)$ to mean $\Sym([n])$. We write $\SymFin(Z)$ for  the group of permutations of $Z$ with finite support, i.e., the subgroup consisting of those $g \in \Sym(Z)$ for which $\{a \in Z \st g(a) \neq a\}$ is finite. 

The \defn{language of set theory} is $\Lang_\in = \{\in\}$ where $\in$ is a binary relation symbol. By a \defn{model of set theory} we mean a transitive model of $\ZFC$ 
that
is a subset of our background model of $\ZFC$. We will use $V$ and its variants to represent models of set theory. For a formula $\varphi$ in set theory, we let $\varphi^{V}$ be the collection of elements $a \in V$ such that $V \models \varphi(a)$. 
(We will sometimes elide mention of $V$ when the model of set theory is clear from context.)

All languages will be relational unless we explicitly say otherwise. We will use $\Lang$ and its variants
to represent languages. 
We will use calligraphic letters for structures in some language and the corresponding Roman letter for their underlying set (e.g., $\cM$ is a structure with underlying set $M$).  
Suppose $\cM$ is an $\Lang$-structure; when 
$\Lang_0 \subseteq \Lang$ and $M_0 \subseteq M$, we write $\cM\rest[M_0]$ to denote the substructure of $\cM$ with underlying set $M_0$, we write $\cM\rest[\Lang_0]$ to denote the reduct of $\cM$ to an $\Lang_0$-structure, and we write $\cM\rest[\Lang_0, M_0]$ to denote $ (\cM\rest[\Lang_0])\rest[M_0]$. 
When $\cA$ and $\cB$ are two $\Lang$-structures, we write $\cA \subseteq \cB$ to mean that $\cA$ is a substructure of $\cB$. We write $\Aut(\cM)$ to denote the automorphism group of a structure $\cM$.

For an infinite 
$\kappa \in \CARD$, we let $\Lkw(\Lang)$ be the smallest collection of formulas containing all atomic $\Lang$-formulas and closed under (1) variable substitutions, (2) adding a new free variable, (3) negations, (4) existential and universal quantification by a single variable, and (5) conjunctions and disjunctions of size $< \kappa$ where all formulas have free variables among a common finite set. 
We let $\Liw(\Lang) = \bigcup_{\kappa \in \CARD} \Lkw(\Lang)$. If $\varphi \in \Liw(\Lang)$, we let $\arity(\varphi)$ denote the \defn{arity} of $\varphi$, i.e., the number of free variables in $\varphi$. If $\varphi \in \Lkw(\Lang)$ is a sentence (i.e., has arity $0$), we write $\models \varphi$ when $\varphi$ is true in all $\Lang$-structures. For a formula $\varphi \in \Liw(\Lang)$ and $\Lang$-structure $\cM$, we let $\varphi^\cM = \{\aa\in M^k \st \cM \models \varphi(\aa)\}$, where $k$ is the arity of $\varphi$.

We call a tuple $\xx$ \defn{non-redundant} if it is injective, considered as a function. For $k \in \w$, we let $\NonRedTuple[k](Z)$ be the collection of non-redundant tuples of elements of $Z$ of size $k$. We let $\NonRed(x_0, \dots, x_{k-1})$ be the first-order formula $\bigwedge_{0 \leq i < j < k}(x_i \neq x_j)$, and say that a formula $\varphi$ of arity $k$ is \defn{non-redundant} when $\models (\forall x_0, \dots, x_{k-1})\, \bigl(\varphi(x_0, \dots, x_{k-1}) \rightarrow \NonRed(x_0, \dots, x_{k-1})\bigr)$, i.e., $\varphi$ holds only of non-redundant tuples. An $\Lang$-structure $\cM$ is \defn{non-redundant} when for all relation symbols $R$ in $\Lang$ of arity $k$,  we have
\[
\cM \models (\forall x_0, \dots, x_{k-1})\, \bigl (R(x_0, \dots, x_{k-1}) \rightarrow \NonRed(x_0, \dots, x_{k-1}) \bigr).
\]
By an \defn{$\Lang$-literal} we mean a quantifier-free $\Lang$-formula 
built from atomic $\Lang$-formulas using only negation.

We say that $\cZ = (Z, \cO_\kappa(Z))$ 
is a \defn{$\kappa$-measurable space} when
\begin{itemize}
\item $\emptyset \in \cO_\kappa(Z) \subseteq \Powerset(Z)$, 

\item $\cO_\kappa(Z)$ is closed under complements,  and

\item $\cO_\kappa(Z)$ is closed under unions of size $\leq \kappa$.
\end{itemize}

For $\kappa \in \CARD$, we let $\sigma_\kappa(\Cantor[Z])$ be the smallest subset of $\Powerset(\Cantor[Z])$ closed under complements, unions of size $\leq \kappa$, 
and containing all sets of the form $\{f \in \Cantor[Z] \st p \subseteq f\}$ where $p \in \Cantor[Z_0]$ for some $Z_0 \in \Powerset_{<\w}(Z)$. 

We say that a probability 
measure $\mu$ on a $\kappa$-measure space $(Z, \cO_\kappa(Z))$ is a \defn{probability $\kappa$-measure} if,
whenever 
$\{A_i\}_{i \in \kappa} \subseteq \cO_\kappa(Z)$ consists of pairwise disjoint sets, we have
$\mu(\bigcup_{i \in \kappa} A_i) = \sum_{i \in \kappa} \mu(A_i)$. Note that a probability \linebreak $\w$-measure is simply a probability measure. 

Suppose $(Z_0, \cO_\kappa(Z_0))$ and $(Z_1, \cO_\kappa(Z_1))$ are $\kappa$-measurable spaces and \linebreak $\alpha\:Z_0 \to Z_1$ is a measurable function, i.e., $\alpha^{-1}(A) \in \cO_\kappa(Z_0)$ for all \linebreak $A \in \cO_\kappa(Z_1)$. For a probability $\kappa$-measure $\mu$ on $(Z_0, \cO_\kappa(Z_0))$, the \defn{pushforward of $\mu$ along $\alpha$}
is the probability $\kappa$-measure 
$\Pushforward[\alpha](\mu)$
on $(Z_1, \cO_\kappa(Z_1))$ defined by
\linebreak
$\Pushforward[\alpha](\mu)(A) = \mu(\alpha^{-1}(A))$
for $A \in \cO_\kappa(Z_1)$.

\subsection{Forcing}      

Our forcing notation 
and terminology
will follow that of \cite{MR756630}. Specifically, when $\mbb{P} = (P, \leq_{\mbb{P}})$ is a partial order, we will interpret $p \leq_{\mbb{P}} q$ 
to mean
that ``$p$ is stronger than $q$''. By a \defn{filter} on $\mbb{P}$ we mean a nonempty set $F \subseteq P$ such that 
\begin{itemize}

\item if $p \leq_{\mbb{P}} q$ and $p \in F$ then $q \in F$, and

\item if $p, q \in F$ then there is an $r \in F$ such that $r \leq_{\mbb{P}}p$ and $r \leq_{\mbb{P}}q$. 
\end{itemize}
When the meaning is clear from context, we will write $\le$ rather than $\le_{\mbb{P}}$ and elide the distinction between $\mbb{P}$ and $P$.

A filter $F$ is \defn{generic} over $V$ (for $\mbb{P}$) when, for all dense $D \subseteq P$ with $D \in V$, we have $F \cap D \neq \emptyset$. (We will typically use some variant of $\Generic$ or $\Generic[H]$ to denote a generic filter.)
If $\varphi$ is a formula in the language of set theory of arity $k$ and $a_0, \dots, a_{k-1}$ are $\mbb{P}$-names and $p \in P$, 
we write $p \forces[\mbb{P}] \varphi(a_0, \dots, a_{k-1})$ when $p$ forces $\varphi(a_0, \dots, a_{k-1})$.

If $\mbb{Q} = (Q, \leq_\mbb{Q})$ is a pre-order, we 
let $x \sim_\mbb{Q} y$ 
denote
the 
relation \linebreak
$(x \leq_\mbb{Q} y) \And (y \leq_\mbb{Q} x)$. We 
let $\mbb{Q}/\!\sim_{\mbb{Q}} = (Q/\!\sim_{\mbb{Q}},\, \leq_\mbb{Q})$ be the corresponding partial order and $\iota_{\mbb{Q}}\:\mbb{Q} \to \mbb{Q}/\!\sim_{\mbb{Q}}$ be the map where $\iota_{\mbb{Q}}(a) = \eqcl[\sim_{\mbb{Q}}](a)$ for all $a\in Q$. Note that (as in \cite{MR756630}) forcing can be developed for pre-orders in an obvious way. For every $\mbb{Q}$-name $a$ we can find a $\mbb{Q}/\!\sim_{\mbb{Q}}$-name $\iota_{\mbb{Q}}(a)$ which is obtained by hereditarily replacing each element $q \in Q$ with an equivalence class $\eqcl[\sim_{\mbb{Q}}](q)$. Then for any $q \in Q$, any formula $\varphi$ in the language of set theory of arity $k$ and any $\mbb{Q}$-names $a_0, \dots, a_{k-1}$, we have 
\[
q \forces[\mbb{Q}]\varphi(a_0, \dots, a_{k-1})\quad\text{ if and only if } \quad \eqcl[\sim_{\mbb{Q}}](q) \forces[\mbb{Q}/\!\sim_{\mbb{Q}}]\varphi(\iota_{\mbb{Q}}(a_0), \dots, \iota_{\mbb{Q}}(a_{k-1})).
\]

\begin{definition}
A forcing notion $\bP$ is \defn{non-trivial} when every condition in $\bP$ has two incompatible extensions.
\end{definition}

\subsection{Boolean Algebras}
Here we recall basic information and definitions regarding Boolean algebras.

\begin{definition}
Suppose $\cA$ and $\cB$ are Boolean algebras with $\cA \subseteq \cB$. We say that $\cA$ is a \defn{regular subalgebra} 
of $\cB$
if whenever $A_0 \subseteq A$ and $a = (\sup A_0)^{\cA}$ (i.e., $a$ is the supremum of $A_0$ in $\cA$), then $a = (\sup A_0)^{\cB}$ (i.e., $a$ is the supremum of $A_0$ in $\cB$). 
\end{definition}

\begin{definition}
Let $Z$ be a collection of sets. We say that a subset $C \subseteq Z$ is 
\defn{unbounded} if
 for every $z \in Z$ there is some $c \in C$ such that $z \subseteq c$. 
The set $C$ is \defn{closed} if whenever $C_0 \subseteq C$ is such that $\bigcup C_0 \in Z$, then also $\bigcup C_0 \in C$.
The set $C$ is \defn{club} if it is closed and unbounded. 
\end{definition}

\begin{definition}
The \defn{density} of a Boolean algebra $\cB$ is the least size of a dense subset 
of $\cB$.
We say that $\cB$ has \defn{uniform density} when for all $b \in B$, the Boolean algebras $\cB$ and 
$\cB\rest[\{ b' \in \cB \st b' \leq b\}]$ 
have the same density.  
\end{definition}

\begin{definition}
Given a non-trivial forcing notion $\bP$, we write $\RO(\bP)$ to denote the unique Boolean completion of $\bP$. 
\end{definition}
Note that there is an embedding $i \colon \bP \to \RO(\bP) \setminus \{0\}$ whose image is dense.

\begin{definition}
Forcing notions $\bP$ and $\bQ$ are \defn{forcing isomorphic}, written $\bP \forcingisom \bQ$, when $\RO(\bP) \isom\RO(\bQ)$.
\end{definition}

If $\bP$ and $\bQ$ are forcing isomorphic, then they give the same forcing extensions of the universe in the sense that for any generic filter $\Generic$
that is $\bP$-generic over $V$, there is a generic filter $\Generic[H]$
which is $\bQ$-generic over $V$ with $V[\Generic] = V[\Generic[H]]$, and vice-versa.

\subsection{Logic Action}      

We now introduce the logic action for a set $Z$. This is an action of $\Sym(Z)$ on the space of $\Lang$-structures with underlying set $Z$. 

\begin{definition}
Let $\Str[\Lang](Z)$ be the collection of $\Lang$-structures with underlying set $Z$. For $\varphi \in \Liw(\Lang)$ of arity $n$ and $a_0, \dots, a_{n-1} \in Z$ we let 
\[
\extent<\Lang>[Z]{\varphi(a_0, \dots, a_{n-1})} = \{\cM \in \Str[\Lang](Z)\st \cM \models \varphi(a_0, \dots, a_{n-1})\}. 
\]
We will simply write $\extent[Z]{\varphi(a_0, \dots, a_{n-1})}$ when the language $\Lang$ is clear from context. 
\end{definition}

\begin{definition} 
Let $\BorelSets<\kappa>[\Lang](Z)$ be the smallest 
subset of $\Powerset(\Str[\Lang](Z))$
that is closed under unions of size $\le\kappa$, closed under complements, and which contains all sets of the form $\extent<\Lang>[Z]{\varphi(a_0, \dots, a_{n-1})}$, where $\varphi$ is a $\Lang$-literal of arity $n$ and $a_0, \dots, a_{n-1} \in Z$. We write $\Str<\kappa>[\Lang](Z)$ to denote the $\kappa$-measurable space $\bigl(\Str[\Lang](Z), \, \BorelSets<\kappa>[\Lang](Z)\bigr)$.
\end{definition}

Note that $\Str<\kappa>[\Lang](Z_0)$ is isomorphic to 
the generalized Cantor space \linebreak
$(\Cantor[Z_1], \sigma_\kappa(\Cantor[Z_1]))$ for some set $Z_1$. 

\begin{definition}
For any $f \in  \Cantor[\coprod_{R \in \Lang}\, Z^{\arity(R)}]$, 
let $\iota_{\Lang, Z}(f)$ be the $\Lang$-structure with underlying set $Z$ such that for $R \in \Lang$ and $a_0, \dots, a_{\arity(R)-1} \in Z$,
we have
\[
\iota_{\Lang, Z}(f) \models R(a_0, \dots, a_{\arity(R)-1})
\text{~~if and only if~~}
f(R, \<a_0, \dots, a_{\arity(R)-1}\>) = 1, 
\]
where $f(R, \<a_0, \dots, a_{\arity(R)-1}\>)$ is the evaluation of $f$ at the obvious 
summand.
\end{definition}

The following is immediate. 
\begin{lemma}
\label{Isomorphism between Str_L and Cantor}
The map $\iota_{\Lang, Z}$ is an isomorphism between \[
(\Cantor[\coprod_{R \in \Lang}\, Z^{\arity(R)}],
\  \sigma_\kappa(\Cantor[\coprod_{R \in \Lang}\, Z^{\arity(R)}]))
\] and $\Str<\kappa>[\Lang](Z)$ as $\kappa$-measurable spaces. 
\end{lemma}

We now define the logic action.

\begin{definition}
Let $\circ_{\Lang, Z}\:\Str[\Lang](Z) \times \Sym(Z) \to \Str[\Lang](Z)$ be the map where, for $g \in \Sym(Z)$, for $\cM \in \Str[\Lang](Z)$, for $R \in \Lang$ of arity $n$, and for $a_0, \dots, a_{n-1} \in Z$, we have
\[
\cM \models R(a_0, \dots, a_{n-1}) 
\text{~~if and only if~~}
\circ_{\Lang, Z}(\cM, g) \models R(g(a_0), \dots, g(a_{n-1})).
\]
We call this map the \defn{logic action} on $\Str[\Lang](Z)$. 
For $g \in \Sym(Z)$, 
when no confusion can occur,
we will abuse notation and also write $g$ to refer to the map $\circ_{\Lang, Z}(\pars, g)$. 
\end{definition}

Note that whenever $g \in \Sym(Z)$ and $A \in \BorelSets<\kappa>[\Lang](Z)$ we have \linebreak $g``[A] \in \BorelSets<\kappa>[\Lang](Z)$ as well, and in particular $\circ_{\Lang, Z}(\pars, g)$ is a measurable map.

\subsection{Borel Codes}      

We now formalize the notion of \emph{Borel codes}.

\begin{definition} 
The collection of \defn{$\kappa$-Borel codes} over $\Str[\Lang](Z)$, denoted $\BorelCodes<\kappa>[\Lang](Z)$, is defined to be the smallest set satisfying the following. 
\begin{itemize}
\item 
$\<0, \varphi, \<a_0, \dots, a_{n-1}\>\> \in \BorelCodes<\kappa>[\Lang](Z)$ 
for every $\varphi \in \Lww(\Lang)$ of arity $n$ that is a finite conjunction of literals and every $a_0,\dots, a_{n-1} \in Z$.

\item 
$\<1, \neg, \gamma\> \in \BorelCodes<\kappa>[\Lang](Z)$
for every $\gamma \in \BorelCodes<\kappa>[\Lang](Z)$.
We will use $\neg \gamma$ to denote $\<1, \neg, \gamma\>$. 

\item 
$\<2, \bigvee, \Gamma\>, \<3, \bigwedge, \Gamma\> \in \BorelCodes<\kappa>[\Lang](Z)$
for every $\Gamma = \{\gamma_i\}_{i \in \kappa} \subseteq \BorelCodes<\kappa>[\Lang](Z)$.
We will use $\bigvee \Gamma$ to denote $\<2, \bigvee, \Gamma\>$ and $\bigwedge \Gamma$ to denote $\<3, \bigwedge, \Gamma\>$. When $\Gamma = \{\zeta, \gamma\}$, we will use $\zeta \And \gamma$ to represent $\bigwedge \{\zeta, \gamma\}$ and $\zeta \Or \gamma$ to represent $\bigvee \{\zeta,\gamma\}$. 
\end{itemize}

When $\kappa = \w$ we will simply write $\BorelCodes[\Lang](Z)$.  The \defn{rank} of a Borel code $\gamma$, denoted $\rank(\gamma)$, is defined as follows. 
\begin{itemize}
\item If $\gamma = \<0, \varphi, \<a_0, \dots, a_{n-1}\>\>$ then $\rank(\gamma)= 0$. 

\item If $\gamma = \neg \zeta$ then $\rank(\gamma) = \rank(\zeta) + 1$. 

\item If $\gamma = \bigwedge \Gamma$ or $\bigvee \Gamma$ then $\rank(\gamma) = \sup \{\rank(\zeta) + 1 \st \zeta \in \Gamma\}$. 

\end{itemize}
\end{definition}

\begin{definition}
We define the \defn{interpretation} of $\gamma \in \BorelCodes<\kappa>[\Lang](Z)$, denoted $\BCvalue(\gamma)$, by induction as follows. 
\begin{itemize}
\item  $\BCvalue(
\<0, \varphi, \<a_0, \dots, a_{n-1}\>\>
) = \extent<\Lang>[Z]{\varphi(a_0, \dots, a_{n-1})}$.

\item $\BCvalue(\neg \gamma) = \Str[\Lang](Z) \setminus \BCvalue(\gamma)$. 

\item $\BCvalue(\bigvee \Gamma) = \bigcup_{\gamma \in \Gamma} \BCvalue(\gamma)$. 

\item $\BCvalue(\bigwedge \Gamma) = \bigcap_{\gamma \in \Gamma} \BCvalue(\gamma)$. 
\end{itemize}
For a probability $\kappa$-measure $\mu$ on $\Str<\kappa>[\Lang](Z)$ and $\gamma \in \BorelCodes<\kappa>[\Lang](Z)$, we say that the \defn{value of $\mu$ on $\gamma$}, written
$\mu(\gamma)$, is $\mu(\BCvalue(\gamma))$.
\end{definition}

Note that $\BorelSets<\kappa>[\Lang](Z)
= 
\{\BCvalue(\gamma) \st \gamma \in \BorelCodes<\kappa>[\Lang](Z)\}$.
We now describe two important collections of Borel codes.
\begin{definition}
The $\kappa$-Borel codes of rank $0$ are called \defn{basic}, and we write $\BasicBC<\kappa>[\Lang](Z)$ to denote the collection of basic $\kappa$-Borel codes.
A $\kappa$-Borel code is \defn{almost basic} when it is of the form $\bigvee \Gamma$ where $\Gamma$ is finite 
and every element of $\Gamma$ is a basic $\kappa$-Borel code; we write $\AlmostBasicBC<\kappa>[\Lang](Z)$ to denote the collection of almost basic $\kappa$-Borel codes.
\end{definition}

As we will see in \cref{Unique-extension}, the values a probability $\kappa$-measure takes on the basic $\kappa$-Borel codes 
suffice to pin down the measure uniquely. 
The interpretations 
of almost basic $\kappa$-Borel codes 
correspond 
to the realizations of quantifier-free first-order formulas. Hence the collection of almost basic $\kappa$-Borel codes is the same in any model of set theory (that contains
the underlying language and the underlying set).

The following is one of the main motivations for considering
probability $\kappa$-measures and $\kappa$-Borel codes 
for $\kappa > \w$, 
as opposed to simply considering 
probability $\w$-measures and $\w$-Borel codes.

\begin{lemma}
\label{Exists Borel codes for formulas}
Suppose $Z, \Lang \in V$
and let
$\varphi \in \Lkpw(\Lang)^V$ be a formula.
Suppose that either 
\begin{itemize}
\item $V \models |Z| \leq \kappa$, or 

\item $\varphi$ is quantifier-free.
\end{itemize}
Then for every tuple $\aa\subseteq Z$ 
there is some code 
$\BCformula({\varphi(\aa)}) \in \BorelCodes<\kappa>[\Lang](Z)$ such that in any model of set theory $V^*$ with $V \subseteq V^*$, we have 
\[
V^* \models \extent<\Lang>[Z]{\varphi(\aa)} = \BCvalue(\BCformula(\varphi(\aa))).
\]
\end{lemma}
\begin{proof}
If $\varphi \in \Lww(\Lang)$ is a literal of arity $n$ and 
$a_0,\dots, a_{n-1} \in Z$, define
$\BCformula(\varphi(a_0, \dots, a_{n-1})) =
\<0, \varphi, \<a_0, \dots, a_{n-1}\>\>$.

Let $F \subseteq \Lkpw(\Lang)$ be the collection of all $\Lkpw(\Lang)$-formulas that satisfy the conclusion for all instantiations by tuples 
from $Z$ of appropriate length. If 
$\varphi \in \Lkpw(\Lang)$ satisfies
$\BCformula(\varphi(a_0, \dots, a_{n-1})) \in \BasicBC<\kappa>[\Lang](Z)$
for all $a_0, \dots, a_{n-1} \in Z$ then it is immediate that $\varphi \in F$. 

Let $\varphi(\xx) \in F$ where $\xx$ consists of distinct variables, let $f\:\xx \to \zz$ be a map, and let $\psi(\zz) = \varphi(\xx / f)$, i.e., the result of substituting the corresponding variables of $\zz$ in place of variables of $\xx$. Then letting $\aa$ be a tuple of elements of $Z$ of length $|\zz|$ (formally, a mapping that assigns each variable in $\zz$ an element of $Z$),
we can take $\BCformula(\psi(\aa)) = \BCformula(\varphi(f \circ \aa))$. Hence $\psi(\zz) \in F$ as well.

If $\varphi \in F$ then we can take $\BCformula(\neg \varphi) = \neg \BCformula(\varphi)$. Similarly, if $\Phi = \{\psi_i\}_{i \in \kappa} \subseteq F$ we can take $\BCformula(\bigvee_{i \in I} \Phi) = \bigvee_{i \in \kappa} \BCformula(\psi_i)$ and $\BCformula(\bigwedge_{i \in I} \Phi) = \bigwedge_{i \in \kappa} \BCformula(\psi_i)$. Therefore $F$ is closed under $\kappa$-sized Boolean operations. Hence $F$ contains all quantifier-free formulas. 

Now suppose $V\models |Z| \leq \kappa$. Further suppose $\varphi = (\exists y)\psi(\xx, y)$. We can take $\BCformula(\varphi) = \bigvee_{a \in Z} \BCformula({\psi(\xx, a)})$. Note that $\BCformula(\varphi) \in \BorelCodes<\kappa>[\Lang](Z)$ as $|Z| \leq \kappa$. Finally, if $\varphi = (\forall y)\psi(\xx, y)$ we can take $\BCformula(\varphi) = \BCformula({\neg (\exists y) \neg \psi(\xx, y)})$. 

Therefore $F$ is closed under variable substitution, $\kappa$-sized Boolean operations, and single-variable existential and universal quantification. Hence, if $V \models |Z| \leq \kappa$, then $F = \Lkpw(\Lang)$.  
\end{proof}

\cref{Exists Borel codes for formulas} motivates the following definition. 

\begin{definition}
Let $\varphi \in \Lkpw(\Lang)$ and let $\mu$ be a probability $\kappa$-measure
on $\Str<\kappa>[\Lang](Z)$. Suppose $\aa \subseteq Z$ is a tuple whose length is the arity of $\varphi$. Whenever either $\varphi$ is quantifier-free or $|Z| \leq \kappa$, then we define $\mu(\varphi(\aa)) = \mu(\BCformula(\varphi(\aa)))$. 
For $|Z| \leq \kappa$,
we define $\Th[\kappa](\mu)$ to be the collection of sentences $\varphi \in \Lkpw(\Lang)$ such that $\mu(\varphi) = 1$. We call $\Th[\kappa](\mu)$ the \defn{almost-sure theory of $\mu$}.
\end{definition}

The following 
provides
a useful way to construct probability $\kappa$-measures on $\Str<\kappa>[\Lang](Z)$.
 
\begin{proposition}
\label{Unique-extension}
Suppose $\mu^\circ \:\BasicBC<\kappa>[\Lang](Z) \to [0, 1]$
satisfies the following conditions.
\begin{itemize}
\item $\mu(\BCformula(\bigwedge \emptyset)) = 1$.

\item If $\varphi$ and $\psi$ are finite conjunctions of $\Lang$-literals 
and $\models \varphi(\xx) \rightarrow \psi(\xx)$ then for any $\aa \in Z$, 
we have
$\mu^\circ(\BCformula(\psi(\aa))) \geq \mu^\circ(\BCformula(\varphi(\aa)))$. 

\item For any $\Lang$-literal $\psi(\xx)$ and $\aa \in Z$,  
we have
$\mu(\psi(\aa) \And \neg \psi(\aa)) = 0$. 

\item For all tuples $\aa \in Z$ and all $\Lang$-literals $\<\psi_i(\xx)\>_{i \in [n+1]}$, 
we have
\[
\mu^\circ\bigg(\bigwedge_{i \in [n]} \psi_i(\aa)\bigg) = \mu^\circ\bigg(\psi_n(\aa) \And \bigwedge_{i \in [n]} \psi_i(\aa)\bigg) + \mu^\circ\bigg(\neg \psi_n(\aa) \And \bigwedge_{i \in [n]} \psi_i(\aa)\bigg).
\]
\end{itemize}
Then there is a unique probability $\kappa$-measure $\mu$ on $\Str<\kappa>[\Lang](Z)$ such that for all $\gamma \in \BasicBC<\kappa>[\Lang](Z)$, we have 
\[
\mu(\gamma) = \mu^\circ(\gamma). 
\]
\end{proposition}
\begin{proof}
This follows from the $\kappa$-measurable version of the \Caratheodory\ extension theorem.
\end{proof}

\begin{definition}
Let $\mu$ be a probability $\kappa$-measure on $\Str<\kappa>[\Lang](Z)$. For \linebreak $\gamma, \zeta \in\BorelCodes<\kappa>[\Lang](Z)$, we write $\gamma \preceq_\mu \zeta$ when $\mu(\gamma \And \neg \zeta) = 0$, and
we write $\gamma \sim_\mu \zeta$ when both $\gamma \preceq_\mu \zeta$ and $\zeta \preceq_\mu \gamma$ hold. 
Define
\[
\RandPO(\mu) = \bigl(\{\gamma \in \BorelCodes<\kappa>[\Lang](Z) \st \mu(\gamma) > 0\},\, \preceq_\mu\bigr).
\]
\end{definition}

It is straightforward to check that $\preceq_\mu$ is a pre-order
 and that $\sim_\mu$ is an equivalence relation. 

We will use the following to transfer
$\kappa$-Borel codes across injections.

\begin{definition}
Let $i\:Z_0 \to Z_1$ be an injection. 
For $\gamma \in \BorelCodes<\kappa>[\Lang](Z_0)$,
define $i[\gamma]\in \BorelCodes<\kappa>[\Lang](Z_1)$ 
by induction as follows. 
\begin{itemize}
\item $i[\<0, \varphi, \<a_0, \dots, a_{n-1}\>\>] = \<0, \varphi, \<i(a_0), \dots, i(a_{n-1})\>\>$. 

\item $i[\neg \varphi] = \neg i[\varphi]$. 

\item $i[\bigwedge \Gamma] = \bigwedge \{i[\gamma] \st \gamma \in \Gamma\}$.

\item $i[\bigvee \Gamma] = \bigvee \{i[\gamma] \st \gamma \in \Gamma\}$.
\end{itemize}
\end{definition}

\subsection{Invariant Measures}
We now review the definition of an invariant probability $\kappa$-measure and the associated notions of dissociatedness, extremality, and ergodicity.
Recall that $X$ is an infinite set.

\begin{definition}
A probability $\kappa$-measure $\mu$ on $\Str<\kappa>[\Lang](X)$ is \defn{invariant} if for all $g \in \Sym(X)$  and all 
$A \in \BorelSets<\kappa>[\Lang](X)$, we have
\[
\mu(A) = \mu(g``[A]).
\]
\end{definition}
For the rest of this paper, we will use the term \emph{invariant measure} to mean an invariant probability $\kappa$-measure. 

We will often be interested in invariant measures that exhibit certain key properties,
namely being \emph{dissociated},
\emph{extreme},
or
\emph{ergodic}.
In the case where $|X| = |\Lang| = |\kappa| = \w$, these three notions are known to be equivalent, 
as we state in Proposition \ref{If countable dissociated equals ergodic equals extreme}.
In \cref{Equivalences with ergodic},
we will show that they are equivalent even without this cardinality constraint.    

\begin{definition}
A probability $\kappa$-measure $\mu$ on $\Str<\kappa>[\Lang](X)$ is \defn{dissociated}
if for all disjoint finite $X_0, X_1 \subseteq X$, all $\gamma_0 \in \BorelCodes<\kappa>[\Lang](X_0)$, and all $\gamma_1 \in \BorelCodes<\kappa>[\Lang](X_1)$,
we have
\[
\mu(\gamma_0 \cap \gamma_1) 
= \mu(\gamma_0) \cdot \mu(\gamma_1).
\]
\end{definition}

In other words, $\kappa$-Borel codes that only refer to disjoint finite index sets are $\mu$-independent.

\begin{definition}
An invariant probability $\kappa$-measure $\mu$
 on $\Str<\kappa>[\Lang](X)$ is \defn{extreme} if, whenever $\nu$ and $\eta$ are distinct invariant probability $\kappa$-measures on $\Str<\kappa>[\Lang](X)$ with $\mu = r \cdot \nu + s \cdot \eta$, then either $r = 0$ or $s = 0$.
\end{definition}

In other words, an invariant measure is extreme if it is not 
a 
convex combination of two distinct invariant measures.

\begin{definition}
Let $\mu$ be an invariant probability $\kappa$-measure
on $\Str<\kappa>[\Lang](X)$ and let $A \in \BorelSets<\kappa>[\Lang](X)$. The set $A$ is \defn{$\mu$-almost invariant} if for all $g \in \Sym(X)$, we have $\mu(g``[A] \SymDiff A) = 0$. For $\gamma \in \BorelCodes<\kappa>[\Lang](X)$ we say that $\gamma$ is \defn{$\mu$-almost invariant} when $\BCvalue(\gamma)$ is.
\end{definition}

\begin{definition}   
An invariant probability $\kappa$-measure
$\mu$ on $\Str<\kappa>[\Lang](X)$ is \defn{ergodic} if, whenever $A \in \BorelSets<\kappa>[\Lang](X)$ is $\mu$-almost invariant, we have $\mu(A) \in \{0, 1\}$.
\end{definition}

The next proposition states that in the countable case, the notions of dissociatedness, extremality, and ergodicity all agree for invariant probability measures.
\begin{proposition}
\label{If countable dissociated equals ergodic equals extreme}
Suppose $|X| = |\Lang| = \w$ and let $\mu$ be an invariant probability measure on $\Str<\w>[\Lang](X)$. The following are equivalent. 
\begin{itemize}
\item $\mu$ is dissociated. 

\item $\mu$ is extreme. 

\item $\mu$ is ergodic. 
\end{itemize}
\end{proposition}
\begin{proof}
See  \cite[Lemma~7.35]{MR2161313} 
as well as \cite[Corollary~2.46]{Aut(M)-invariance}.
\end{proof}

One of the reasons why ergodic invariant probability measures are so important is that, in the countable context, 
every invariant probability measure is a mixture of ergodic ones.

\begin{proposition}[\hbox{\cite[Theorem~A1.3]{MR2161313}}]
\label{prop:kal-ergodic}
Suppose $|X| = |\Lang| = \w$ and let $\mu$ be an invariant probability measure on $\Str<\w>[\Lang](X)$. Then $\mu$ is a mixture
of ergodic invariant probability measures on $\Str<\w>[\Lang](X)$. 
\end{proposition}

\subsection{\Fraisse\ Theory}

We now recall some basic notions of \Fraisse\ theory. For more on this, see \cite[Chapter~7.1]{MR1221741}.

\begin{definition}
Suppose $\AgeK$ is a collection of finite $\Lang$-structures that is closed under isomorphism. 

We say that $\AgeK$ has the \defn{hereditary property \HP}\ if, whenever $\cM_1 \in \AgeK$ and $M_0 \subseteq M_1$, 
we have $\cM_1\rest[M_0] \in \AgeK$. 

We say that
$\AgeK$ has the \defn{joint embedding property \JEP}\ if, whenever $\cM_0, \cM_1 \in \AgeK$, there exist an $\cM_2\in \AgeK$ and embeddings $i\:\cM_0 \to \cM_2$ and $j\:\cM_1\to \cM_2$. 

We say that
$\AgeK$ has the \defn{amalgamation property \AP}\ if, whenever \linebreak $\cM_0, \cM_1, \cM_2 \in \AgeK$ and the maps $i_1\:\cM_0 \to \cM_1$ and $i_2\:\cM_0 \to \cM_2$ are embeddings, 
there 
exist $\cM_3 \in \AgeK$ and embeddings $j_1\:\cM_1\to \cM_3$ and $j_2\:\cM_2 \to \cM_3$ such that $j_1 \circ i_1 = j_2 \circ i_2$. 

If, further, $j_1, j_2$ can be chosen such that 
\[
j_1``[\cM_1] \cap j_2``[\cM_2] = (j_1 \circ i_1)``[\cM_0],
\]
then we say that $\AgeK$ has the \defn{strong amalgamation property \SAP}.

We say that $\AgeK$ is a \defn{\Fraisse\ class} (in $\Lang$) when it satisfies \HP, \JEP, and \AP. A \defn{strong \Fraisse\ class} is a \Fraisse\ class that satisfies \SAP. 
\end{definition}

\begin{definition}
Let $\cM$ be an $\Lang$-structure.
The \defn{age} of $\cM$ is the collection of those
finite $\Lang$-structures that embed into $\cM$. 
\end{definition}

\begin{definition}
Suppose $\AgeK$ is a \Fraisse\ class in $\Lang$. A \defn{\Fraisse\ limit}
of $\AgeK$ is an $\Lang$-structure $\cM$ such that 
\begin{itemize}
\item for every $\cN \in \AgeK$ there is an embedding from $\cN$ into $\cM$, and

\item if $\cN \in \AgeK$ and $i, j\:\cN \to \cM$ are embeddings then there is a $g \in \Aut(\cM)$ 
such that $i = g \circ j$. 

\end{itemize}
\end{definition}

The following is standard. 

\begin{proposition}
Let $\AgeK$ be a collection of finite $\Lang$-structures that is closed under isomorphism.
Suppose $\AgeK$ has only countably many isomorphism types. 
Then we have the following.
\begin{itemize}
\item $\AgeK$ is the age of an $\Lang$-structure if and only if $\AgeK$ has \HP\ and \JEP. 

\item $\AgeK$ has a countably infinite \Fraisse\ limit 
if and only if $\AgeK$ has \HP, \JEP, and \AP, i.e., $\AgeK$ is a \Fraisse\ class. 
\end{itemize}

\end{proposition}

\begin{definition}
A strong \Fraisse\ class 
$\AgeK$ in $\Lang$ is \defn{trivial} if for all literals $\psi$ in $\Lang$ of arity $n$, either (i) 
for all $\cA \in \AgeK$, we have
\[
\cA \models (\forall x_0, \dots x_{n-1})\, \bigl(\NonRed(x_0, \dots, x_{n-1}) \rightarrow \psi(x_0, \dots, x_{n-1})\bigr),
\]
or  (ii) 
for all $\cA \in \AgeK$, we have
\[
\cA \models (\forall x_0, \dots x_{n-1})\, \bigl(\NonRed(x_0, \dots, x_{n-1}) \rightarrow \neg \psi(x_0, \dots, x_{n-1})\bigr).
\]
A structure is \defn{trivial} when its age is a trivial strong \Fraisse\ class. 
\end{definition}

\begin{remark}
For the rest of the paper, unless we explicitly mention otherwise, we will assume that all strong \Fraisse\ classes are non-trivial and that all infinite structures 
are non-trivial. 
\end{remark}

\section{Cohen Forcing and Random Forcing on $\Cantor[X]$}
\label{Section: Examples on Cantor} 

In this section, we review properties of Cohen forcings and random forcings (with Lebesgue measure) over arbitrary sets.
Both of these can be seen as forcing with specific subsets of $\Cantor[X]$.

Let $\LangOne$ be the language with a single unary relation symbol $U_1$. By \cref{Isomorphism between Str_L and Cantor}, there is an isomorphism $\iota_{\LangOne, X}$ from $(\Cantor[X], \sigma_\kappa(\Cantor[X]))$ to $\Str<\kappa>[\LangOne](X)$. 
In general, we will not distinguish between $\Cantor[X]$ and $\Str[\LangOne](X)$ and we will abuse notation in the obvious way by referring to a $\kappa$-Borel code in $\Str<\kappa>[\Lang_1](X)$ as a $\kappa$-Borel code in $\Cantor[X]$, and likewise abuse notation by referring to the interpretation of such a $\kappa$-Borel code as a subset of $\Cantor[X]$. 

Since $\Cantor[X]$ and $\Str[\LangOne](X)$ are isomorphic,
Cohen forcing and random forcing with Lebesgue measure can also be thought of as forcing with subsets of $\Str[\LangOne](X)$. Hence Cohen forcings and random forcings with Lebesgue measure  on $\Cantor[X]$
are special cases, respectively, of forcing with a strong \Fraisse\ class and forcing with an invariant measure on the space of structures with a given underlying set $X$.

\subsection{Cohen Forcing}
The partial order for Cohen forcing over a set $X$, denoted  $\CantorFiniteSeq[X]$, is defined to be
\[
\CantorFiniteSeq[X]=\bigl\{p\st (\exists X_0 \in \PowersetFin(X))\  p \in \Cantor[X_0]\bigr\},
\]
ordered by reverse inclusion. Suppose $\Generic$ is $\CantorFiniteSeq[X]$-generic over $V$ and let \linebreak $\GenericUnion=\bigcup_{p \in \Generic}p.$ Then $\GenericUnion$ is a function from $X$
into $\{0, 1\}$ and $V[\Generic]=V[\GenericUnion]$. The function $\GenericUnion$ is called a \defn{Cohen generic function} and we say that $\GenericUnion$ is \defn{Cohen over} $V$. 
Note that $\GenericUnion$
has a canonical $\CantorFiniteSeq[X]$-name $\name{\GenericUnion}$, defined by
\[
\name{\GenericUnion} = \{(p, \name{n}) \st n \in X,\,\  p \in \CantorFiniteSeq[X],\text{ and } p(n) = 1\}.
\]
When $X = \w$ we call $\GenericUnion$ a \defn{Cohen generic real}.

The following is standard. 
\begin{lemma}
[\hbox{\cite[Lemma 14.35]{MR1940513}}]
	\label{chain condition lemma for cohen}
	The partial order $\CantorFiniteSeq[X]$ has c.c.c.
\end{lemma}

We will be interested in understanding when a forcing notion is isomorphic to Cohen forcing over a cardinal. 

\begin{definition}
A partial order $\bP$ is \defn{Cohen} when $\RO(\bP) \cong \CantorFiniteSeq[X]$ for some infinite set $X$.
\end{definition}

The following results will allow us to determine when a partial order is Cohen.  

\begin{proposition}[\hbox{\cite[Comment after Theorem 30.10]{MR1940513}}]
\label{lem:countable_forcing_is_cohen} 
If $\bP$ is a non-trivial countable forcing notion, then $\bP$ is forcing isomorphic to $\CantorFiniteSeq[\w]$. \end{proposition}

\begin{theorem}[\hbox{\cite[Theorem 30.10]{MR1940513}}]
\label{Regular subalgebra definition of being Cohen}
Suppose $\bB$ is an infinite Boolean algebra of uniform density. Then $\cB$ is Cohen if and only if the set 
\[
\{\bA \subseteq \bB \st |A| \leq \w\text{ and }\bA\text{ is a regular subalgebra of } \bB\}
\]
has a club subset $C$ such that whenever $\cA_0, \cA_1 \in C$, if $\cA_2$ is the Boolean algebra generated 
(as a Boolean algebra)
by $A_0 \cup A_1$, 
then $\cA_2 \in C$. Further, if the uniform density of $\bB$
is $\kappa$, then $\cB \cong \CantorFiniteSeq[\kappa]$.
\end{theorem}

\subsection{Lebesgue $\kappa$-measure}      

While we will eventually be interested in arbitrary $\Sym(X)$-invariant probability $\kappa$-measures on $\Str<\kappa>[\Lang](X)$, there is one specific measure on $\Cantor[X]$ which is especially important and provides us with a good example for what is to follow.

\begin{definition}
\label{Definition of Lebesgue}
Define the \defn{Lebesgue $\kappa$-measure} on $(\Cantor[X], \sigma_\kappa(\Cantor[X]))$ to be the 
probability $\kappa$-measure
$\Lebesgue<\kappa>[X]$ such that for any $X_0 \in \PowersetFin(X)$ and $\tau\in \Cantor[X_0]$, we have\[
\Lebesgue<\kappa>[X]\bigl(\{f \in \Cantor[X] \st \tau\subseteq f\}\bigr) = 2^{-|X_0|}. 
\]
\end{definition}

We will abuse notation and use $\Lebesgue<\kappa>[X]$ to denote the corresponding 
probability $\kappa$-measure on $\Str[\LangOne](X)$, i.e., the pushforward of $\Lebesgue<\kappa>[X]$ along $\iota_{\LangOne, X}$.

Note that by \cref{Unique-extension}, there is a unique probability $\kappa$-measure satisfying \cref{Definition of Lebesgue}.

\begin{definition}
For $\gamma \in \BorelCodes<\kappa>[\LangOne](X)$ we write $\Lebesgue<\kappa>[X](\gamma)$ to mean $\Lebesgue<\kappa>[X](\BCvalue(\gamma))$. 
\end{definition}

We have the following immediate result. 
\begin{lemma}
The measure $\Lebesgue<\kappa>[X]$ is $\Sym(X)$-invariant. 
\end{lemma}

Suppose $V_0 \subseteq V_1$  and suppose $\gamma \in [\BorelCodes<\kappa>[\LangOne](X)]^{V_0}$ and $Y = \BCvalue(\gamma)^{V_0}$. In general, we do not have that $\bigl[\Lebesgue<\kappa>[X](Y)\bigr]^{V_0} = \bigl[\Lebesgue<\kappa>[X](Y)\bigr]^{V_1}$. In fact, if $[2^X]^{V_0}$ is countable in $V_1$ then we have $\bigl[\Lebesgue<\kappa>[X]((2^X)^{V_0})\bigr]^{V_1} = 0$. 

The purpose of using $\kappa$-Borel codes instead of subsets when studying 
Lebesgue $\kappa$-measure is that the Lebesgue $\kappa$-measure of a $\kappa$-Borel code is independent of the model of set theory in which we are working, as the following result 
asserts. (We will not prove this here as it is a special case of \cref{Relativizations give same values} applied to $\Str[\LangOne](X)$, and we will not use this specific version beyond our discussion here.)

\begin{proposition}
\label{Lebesgue measure of Borel codes is absolute}
Suppose $V_0 \subseteq V_1$, let $X \in V_0$, and suppose
$\gamma \in [\BorelCodes<\kappa>[\LangOne](X)]^{V_0}$.  
Then for any $r \in [0, 1]^{V_0}$, we have
\[
V_0 \models \Lebesgue<\kappa>[X](\gamma) = r
\qquad
\text{if and only if}
\qquad
V_1 \models \Lebesgue<\kappa>[X](\gamma) = r.
\]
\end{proposition}

In particular, if $V_0 \subseteq V_1$ and $X \in V_0$ and $\gamma_0, \gamma_1 \in [\BorelCodes<\kappa>[\LangOne](X)]^{V_0}$ then \linebreak $V_0 \models \gamma_0 \preceq_{\Lebesgue<\kappa>[X]} \gamma_1$ if and only if $V_1 \models \gamma_0 \preceq_{\Lebesgue<\kappa>[X]} \gamma_1$. Hence we have \linebreak $\bigl[\RandPO(\Lebesgue<\kappa>[X])\bigr]^{V_0} \subseteq \bigl[\RandPO(\Lebesgue<\kappa>[X])\bigr]^{V_1}$ 
as partial order.

\subsection{Random Forcing} 
We now recall facts about random forcing. The relationship between random forcing and forcing with invariant measures is similar to the relationship between Cohen forcing and forcing with strong \Fraisse\ classes. 

Let $\RRF<\kappa>(X) = \RandPO(\Lebesgue<\kappa>[X])$. We call forcing with $\RRF<\kappa>(X)$ \defn{random $\kappa$-forcing} over $X$. 
Let $\name{r}$ be the $\RRF<\kappa>(X)$-name 
\begin{align*}
\name{r} = \bigg\{\Bigl(\BCformula(\bigwedge_{i \in [k]} \psi(x_i)),\, \name{n}\Bigr) &\st \{x_i\}_{i \in [k]} \subseteq X \ \text{and}\  \bigwedge_{i \in [k]} \bigl(\psi(x_i) \rightarrow U(n)\bigr) \ 
\\
&\hspace*{15pt}
\text{and}\  
\neg \bigl(\bigwedge_{i \in [k]}\psi(x_i) \rightarrow \bot\bigr)\bigg\}.
\end{align*}
If $\Generic$ is generic for $\RRF<\kappa>(X)$ then we call $\name[\Generic]{r}$ the \defn{random function} associated to $\Generic$ and we say $\name[\Generic]{r}$ is \defn{random over $V$}. Further, $\name{r}$ is the canonical name of such a function.  When $X = \w$, we call $\name[\Generic]{r}$ a \defn{random real} (over $V$). 

The following lemmas are standard; see, e.g., 
\cite[Lemma 3.1.1]{bartoszynski1995set} and 
\cite[Lemma 8.8]{schindler2014set}.
\begin{lemma}
	\label{chain condition lemma for random}
	$\RRF<\kappa>(X)$ is c.c.c.
\end{lemma}

\begin{lemma}
If $\Generic$ is generic for $\RRF<\kappa>(X)$ then $V[\Generic] = V\bigl[\name[\Generic]{r}\bigr]$. 
\end{lemma}

\subsection{$\kappa$-Measurable Space Equivalence} 

In this paper we are interested in the connections between invariant measures on different spaces of structures as well as the structures which come from random forcing with invariant measures. However, if we are only interested in the space of structures as a $\kappa$-measurable space, then all classes of structures with languages and sets of the same size are cryptomorphic. The following immediate lemma, along with \cref{Isomorphism between Str_L and Cantor}, expresses this fact. 

\begin{lemma}
\label{Isomorphism between Borel codes on Str_L and Cantor}
There is a bijection $\iota_{\BorelCodes<\kappa>[\Lang](X)}\: \BorelCodes<\kappa>[\LangOne](\coprod_{R \in \Lang}\, X^{\arity(R)}) \to \BorelCodes<\kappa>[\Lang](X)$ such that for every $\gamma \in \BorelCodes<\kappa>[\LangOne](\coprod_{R \in \Lang}\, X^{\arity(R)})$, we have
\begin{itemize}
\item $\rank(\gamma) = \rank\bigl(\iota_{\BorelCodes<\kappa>[\Lang](X)}(\gamma)\bigr)$ and

\vspace*{5pt}
\item $\iota_{\Lang, X}``[\BCvalue(\gamma)] = \BCvalue(\iota_{\BorelCodes<\kappa>[\Lang](X)}(\gamma))$,
\end{itemize}
where $\iota_{\Lang, X}$ is as in \cref{Isomorphism between Str_L and Cantor}.
\end{lemma}

\section{Forcing with Strong \Fraisse\ Classes}
\label{Section: Forcing with Strong Fraisse Classes}

We now recall from \cite{Golshani} and \cite{Cohen} the notion of a Cohen generic structure as well as some properties of forcing with a strong \Fraisse\ class.

\begin{definition}
Let $\AgeK$ 
be a strong \Fraisse\ class. Define $\Cohen[X](\AgeK)$ to be the partial order consisting of elements of $\AgeK$ whose underlying set
is a subset of $X$, with $\leq$ defined by
$\cA \leq \cB$ if and only if $\cB \subseteq \cA$.
We say that a structure $\GenericUnion$ is \defn{Cohen generic} for $\AgeK$ on $X$ over $V$ if there is a generic $\Generic$ for $\Cohen[X](\AgeK)$ such that $\GenericUnion = \bigcup \Generic$. 
\end{definition}

The following is immediate. 
\begin{lemma}
A strong \Fraisse\ class is non-trivial if and only if $\Cohen[X](\AgeK)$ is non-trivial as a forcing notion.
\end{lemma}

The following is 
a key
result regarding Cohen generic structures. 
\begin{lemma}[\hbox{\cite[Lemma~2.3]{Golshani}}]
If $\AgeK$ is a strong \Fraisse\ class then $\Cohen[X](\AgeK)$ has c.c.c. 
\end{lemma}

Given two strong \Fraisse\ classes $\AgeK_0$ and $\AgeK_1$ in $V$ and a set $X \in V$, 
a natural
question to consider
is: ``When does the existence of a generic over $V$ for $\AgeK_0$ ensure the existence of a generic over $V$ for $\AgeK_1$?'' The following results give a partial answer.

\begin{proposition}
\label{All Cohen structures define Cohen real}
If $\AgeK$ is a (non-trivial) strong \Fraisse\ class over $\Lang$ and $\Generic$ is generic for $\Cohen[X](\AgeK)$ over $V$ then there is an $\Generic[H]$ which is generic for $\CantorFiniteSeq[X]$ such that $\Generic[H] \in V[\Generic]$. 
\end{proposition}
\begin{proof}
As $\AgeK$ is non-trivial and has \HP\ there exist 
\begin{itemize}
\item an $n \in \w$,

\item an atomic formula $\psi \in \Lang$ of arity $n$, 

\item $\cA, \cB \in \AgeK$, and

\item enumerations (without repetitions) $\<a_i\>_{i \in [n]}$  of $A$ and  $\<b_i\>_{i \in [n]}$ of $B$ 
\end{itemize}
for which
$\cA \models \psi(a_0, \dots, a_{n-1})$ and $\cB \models \neg \psi(b_0, \dots, b_{n-1})$. 

For $\alpha \in X$ let $\cA_\alpha$ be a structure where the map $a_j \mapsto (\alpha, j)$ is an isomorphism from $\cA$ to $\cA_\alpha$.
Define $\cB_\alpha$ similarly. 
Let $\iota\:X \times [n] \to X$ be a bijection. 

Suppose $\cM \in \AgeK$ with $M \subseteq X$. 
Let
$Z_{\cM} = \{\ell \in X \st
\iota``[\{\ell\} \times [n]] \subseteq M
\}$,
and let
$p_{\cM}\:Z_{\cM} \to \{0, 1\}$ be the map defined by
\[
p_{\cM}(\ell)  
 = 
 \begin{cases}
1 & \text{if } \cM \models \psi\bigl(\iota(\<\ell, 0\>), \iota(\<\ell, 1\>), \dots, \iota(\<\ell, n-1\>)\bigr)\\
0 & \text{otherwise}.
 \end{cases}
\]
If $\Generic$ is generic for $\Cohen[X](\AgeK)$ we let $\Generic[H]= \{p_{\cM} \st \cM \in \Generic\}$. It suffices to prove the following claim. 
\begin{claim}
$\Generic[H]$ is generic for $\CantorFiniteSeq[X]$.
\end{claim}
\begin{proof}
First note that $\Generic[H]$ is a filter as $\Generic$ is a filter and $\AgeK$ has \HP. Now suppose $D \subseteq \CantorFiniteSeq[X]$ is dense. Let $D^* = \{ \cM \in \Cohen[X](\AgeK) \st p_{\cM} \in D\}$. 
Let $\cN \in \Cohen[X](\AgeK)$. Let $E = \{\cM \in \AgeK \st (\exists Y \subseteq X)\ \,M = \iota``[Y \times [n]]\}$. The set $E$ is dense and so we can find an $\cN^* \in \Generic \cap E$ such that $\cN \subseteq \cN^*$. 

Then there must be some $q \in D$ such that $p_{\cN^*} \subseteq q$. Let $\<x_i\>_{i \in [s]}$ be an enumeration of $\dom(q) \setminus \dom(p_{\cN^*})$. 

Let $\cN_{-1} = \cN^*$. For $i \in [s]$ let $\cN_i^0$ be a structure on $N_{i-1} \cup \{x_i\} \times [n]$ containing $\cN_{i-1}$ and $\cA_{x_i}$, and let $\cN_i^1$ be a structure on $N_{i-1} \cup \{x_i\} \times [n]$ containing $\cN_{i-1}$ and $\cB_{x_i}$. Note that we can always find such structures as $\AgeK$ has \JEP\ and \HP. Let $\cN_i = \cN_i^{q(x_i)}$. 

We then have $p_{\cN_{s-1}} = q$, and so $\cN_{s-1} \in D^*$. Therefore $\Generic[H] \cap D \neq \emptyset$. As $D$ was an arbitrary dense subset of $\CantorFiniteSeq[X]$, this proves the claim. 
\end{proof}

This completes the proof of
\cref{All Cohen structures define Cohen real}.
\end{proof}

The following is one of two partial converses to \cref{All Cohen structures define Cohen real} that we obtain.

\begin{proposition}
\label{All Fraisse classes give the same universe}
Let
$\AgeK_0$ and $\AgeK_1$ be (non-trivial) strong \Fraisse\ classes over $\Lang$ with only countably many isomorphism types, and
suppose
$\Generic_0$ is generic for $\Cohen[\w](\AgeK_0)$. 
Then there is a $\Generic_1 \in V[\Generic_0]$ which is generic for $\Cohen[\w](\AgeK_1)$. 
\end{proposition}
\begin{proof}
By \cref{lem:countable_forcing_is_cohen}, we have
$\Cohen[\w](\AgeK_0) \forcingisom  \CantorFiniteSeq[\w] \forcingisom \Cohen[\w](\AgeK_1)$,
from which the proposition follows.
\end{proof}

\begin{lemma}
Suppose $X$ is a set and $\AgeK$ is a strong \Fraisse\ class with at most $|X|$-many isomorphism classes. Then $\RO(\Cohen[X](\AgeK))$ has uniform density $|X|$. 
\end{lemma}
\begin{proof}
This is clear by our assumption on $\AgeK$.
\end{proof}

We then have our second partial converse to \cref{All Cohen structures define Cohen real}.

\begin{proposition}
\label{prop:second-partial-converse}
Suppose $\AgeK$ is a strong \Fraisse\ class with at most $\w_1$-many elements up to isomorphism. Then $\Cohen[\w_1](\AgeK)$ and $\CantorFiniteSeq[\w_1]$ are forcing isomorphic, and in particular $\Cohen[\w_1](\AgeK)$ is Cohen.
\end{proposition}
\begin{proof}
Let $\mbb{B}$ be the 
Boolean
subalgebra of $\RO(\Cohen[X](\AgeK))$ generated by $\Cohen[X](\AgeK)$. Then 
$\mbb{B} \forcingisom \Cohen[X](\AgeK)$, so it suffices to show that $\mbb{B} \forcingisom  \CantorFiniteSeq[\w_1].$
Let \linebreak $C = \{i_\alpha``[\Cohen[\alpha](\AgeK)] \st \alpha < \w_1\}$, where $i_\alpha$ is the canonical embedding of $\Cohen[\alpha](\AgeK)$ into $\mbb{B}$.
Then $C$ witnesses the club requested in
\cref{Regular subalgebra definition of being Cohen} for $\mbb{B}$.
\end{proof}

\begin{remark}
\label{old-Kostana-remark}
Note that 
Propositions~\ref{All Fraisse classes give the same universe} and
\ref{prop:second-partial-converse}
fully answer \cite[Question~1]{Cohen} when the underlying set has size at most $\w_1$.
\end{remark}

\section{Lifting Invariant Measures}
\label{Section: Invariant Measures}

\subsection{Liftings and Relativizations of Invariant Measures}
\label{subsec:liftings}

We begin by proving some basic facts about invariant measures on $\Str<\kappa>[\Lang](X)$.

\begin{definition}  
Suppose $f\:X \to Y$ and $\Lang_0 \subseteq \Lang$. Then we define \linebreak $\Proj<f>[\Lang_0]\: \Str<\kappa>[\Lang](Y) \to \Str<\kappa>[\Lang_0](X)$ to be the map which takes an $\Lang$-structure $\cM$ 
with underlying set $Y$
to the $\Lang_0$-structure $\cN$ 
with underlying set $X$ such that 
\[
\cN \models R(x_0, \dots, x_{n-1}) 
\qquad
\text{if and only if}
\qquad
\cM \models R(f(x_0), \dots, f(x_{n-1}))
\] 
for every $R \in \Lang_0$ of arity $n$ and $x_0, \dots, x_{n-1} \in X$. 
For $X \subseteq Y$, 
we will write $\Proj<X>[\Lang_0]$ to mean $\Proj<f>[\Lang_0]$,
where $f\colon X\to Y$ is the inclusion map. 
\end{definition}

\begin{definition}
Let $\mu$ 
be a probability $\kappa$-measure on $\Str<\kappa>[\Lang](Y)$.
For $f\:X \to Y$ and $\Lang_0 \subseteq \Lang$, we write $\Proj*<f>[\Lang_0](\mu)$ to denote the ``pushforward'' of $\mu$ along $\Proj<f>[\Lang_0]$, i.e., the probability $\kappa$-measure on $\Str<\kappa>[\Lang_0](X)$ where
\[
\Proj*<f>[\Lang_0](\mu)(A) = \mu(\Proj<f>[\Lang_0]^{-1}(A))
\]
for every  $A \in \BorelSets<\kappa>[\Lang_0](X)$.
\end{definition}

Note that a probability $\kappa$-measure $\mu$ on $\Str<\kappa>[\Lang](X)$ is invariant if and only if $\Proj*<g>[\Lang](\mu) = \mu$
for all $g \in \Sym(X)$.  
In fact, we have the following.

\begin{lemma}
\label{Finite support invariance gives full invariance}
Let $\mu$ be a probability $\kappa$-measure on $\Str<\kappa>[\Lang](X)$.
Suppose that 
$\Pushforward[\alpha](\mu) = \mu$
for every $\alpha \in \SymFin(X)$.
Then $\mu$ is invariant. 
\end{lemma}
\begin{proof}
Suppose $\beta \in \Sym(X)$ and $\gamma \in \BasicBC<\kappa>[\Lang](X)$. There is some $\beta_- \in \SymFin(X)$ such that $\beta(\gamma) = \beta_-(\gamma)$. Then $\mu(\gamma) = \mu(\beta_-(\gamma)) = \mu(\beta(\gamma))$. Therefore, by  \cref{Unique-extension}, we have $\mu = \Proj*<\beta>[\Lang](\mu)$.  
\end{proof}

The following immediate fact is 
a key property
of invariant measures. 

\begin{lemma}
\label{Invariant measures have unique finite projections}
Suppose
$\Lang_0 \subseteq \Lang$, 
and let $n \in \w$.
Suppose $i_0, i_1\:[n] \to X$ are injective. Then for any invariant measure $\mu$ on $\Str<\kappa>[\Lang](X)$, we have $\Proj*<i_0>[\Lang_0](\mu) = \Proj*<i_1>[\Lang_0](\mu)$. 
\end{lemma}

\cref{Invariant measures have unique finite projections} is 
useful
because it will allow us to uniquely ``lift'' an invariant measure on one infinite set to an invariant measure on any other infinite set.

\begin{definition}
Suppose $\mu$ is an invariant measure on $\Str<\kappa>[\Lang](X)$. For $n \in \w$, let $i_n\:[n] \to X$ be 
an
injection. We define the \defn{skeleton}
of $\mu$ to be 
\[
\Skeleton(\mu) = \<\Proj*<i_n>[\Lang_0](\mu)\>_{n \in \w, \, \Lang_0 \in \PowersetFin(\Lang)}.
\]
We let $\Skeleton<\Lang_0>[n](\mu) = \Proj*<i_n>[\Lang_0](\mu)$. 
\end{definition}

\begin{proposition}
\label{Invariant-measures-lift}
For any invariant measure $\mu$ on $\Str<\kappa>[\Lang](Y)$, there is a unique invariant measure $\nu$ on $\Str<\kappa>[\Lang](X)$ such that $\Skeleton(\mu) = \Skeleton(\nu)$. 
\end{proposition}
\begin{proof}
Suppose $\varphi(a_0, \dots, a_{n-1})$ is a finite conjunction of literals in $\Lww(\Lang_0)$, where $\Lang_0 \subseteq \Lang$, with parameters in $X$. Let $X_0 = \{a_i\}_{i \in [n]}$, let $k = |X_0|$, and let $\zeta\:X_0 \to [k]$ 
be a bijection. Let \[
\nu^\circ\bigl(\varphi(a_0, \dots, a_{n-1})\bigr) = \Skeleton<\Lang_0>[k](\mu)\bigl(\BCformula(\varphi(\zeta(a_0), \dots, \zeta(a_{n-1})))\bigr)
.
\]

Note that as $\mu$ is invariant, the value of $\nu^\circ$ does not depend on the bijection $\zeta$. But then by \cref{Unique-extension} there is a unique probability \linebreak
$\kappa$-measure $\nu$ on $\Str<\kappa>[\Lang](X)$ such that 
$\nu^\circ(\gamma) = \nu(\gamma)$
for any $\gamma \in \BasicBC<\kappa>[\Lang](X)$.

For any $\alpha$ with finite support, we have $\Pushforward[\alpha](\nu) = \nu$. Hence $\nu$ is invariant by \cref{Finite support invariance gives full invariance}.

It is immediate from 
the way
$\nu$ was constructed that $\Skeleton(\mu) = \Skeleton(\nu)$. 
\end{proof}

\begin{definition}
When $\mu$ is an invariant measure on $\Str<\kappa>[\Lang](Y)$ we write $\mu_X$ for the unique 
invariant measure
on $\Str<\kappa>[\Lang](X)$ such that $\Skeleton(\mu) = \Skeleton(\mu_X)$, and call $\mu_X$ the \defn{lifting} of $\mu$ to $X$. 
\end{definition}

As a consequence of \cref{Invariant-measures-lift}, for any invariant measure $\mu$ on $\Str<\kappa>[\Lang](X)$ and any injection $i\:\w \to X$, the measure $\mu$ is completely determined by $\Proj*<i>[\Lang](\mu)$. 

The following is also an immediate consequence of \cref{Invariant-measures-lift}.
\begin{corollary}
Let $f\:X \to Y$ be an injection and let $\mu$ be an invariant measure on $\Str<\kappa>[\Lang](Y)$. Then $\Proj*<f>[\Lang](\mu) = \mu_X$. 
\end{corollary}

Because we will want to move between different models of set theory, it will be important to know, for $V_0 \subseteq V_1$, when a probability $\kappa$-measure $\mu \in V_0$ is the same as a probability $\kappa$-measure $\mu^* \in V_1$. The general notion of relativization is defined in \cite{MR2652199}, but for our purposes the following definition suffices. 
(We will generally use square brackets when we need to parenthesize objects being relativized.)

\begin{definition}
Suppose $V_0 \subseteq V_1$ and suppose that for 
$i \in \{0, 1\}$, we are given
a probability
$\kappa$-measure 
$\mu_i \in V_i$ 
on $[\Str<\kappa>[\Lang](X)]^{V_i}$. 
We say that $\mu_1$ is the \defn{relativization} of $\mu_0$ to $V_1$ if 
\[
\mu_0(\gamma) = \mu_1(\gamma)
\]
for all $\gamma \in [\BasicBC<\kappa>[\Lang](X)]^{V_0}$,
and in this case we write $[\mu_0]^{V_1}$ for $\mu_1$. 
\end{definition}

Note that by \cref{Unique-extension} and the fact that $[\BasicBC<\kappa>[\Lang](X)]^{V_0} = [\BasicBC<\kappa>[\Lang](X)]^{V_1}$, any probability $\kappa$-measure on $\Str<\kappa>[\Lang](X)$ in $V_0$ has a unique relativization to $V_1$.

The following shows that a measure and its relativization always agree on the measure of any $\kappa$-Borel code in the domain of the original measure.

\begin{theorem}
\label{Relativizations give same values}
Suppose
$V_0 \subseteq V_1$.
Let $\kappa, X, \Lang \in V_0$
and let
$\gamma \in [\BorelCodes<\kappa>[\Lang](X)]^{V_0}$.
Suppose
$\mu \in V_0$ is a probability $\kappa$-measure on $\Str<\kappa>[\Lang](X)$.
Then for any \linebreak $r \in [0, 1]^{V_0}$, we have
\[
V_0 \models \mu(\gamma) = r
\qquad
\text{if and only if}
\qquad
V_1 \models \mu^{V_1}(\gamma) = r
.
\]
\end{theorem}
\begin{proof}
For $\gamma, \zeta \in \BorelCodes<\kappa>[\Lang](X)$ let $\gamma \SymDiff \zeta = (\gamma \And \neg \zeta) \Or (\neg \gamma \And \zeta)$. 
Let $\cV$ be a collection of models of set theory such that 
\begin{itemize}
\item 
$V_0 \subseteq V$
for all $V \in \cV$, 

\item $V_0, V_1 \in \cV$, and

\item there is a $V \in \cV$ such that $V \, \models \, |X| = |\Lang| = |\kappa| = \w$. 
\end{itemize}

Note that for $V^+, V^- \in \cV$ we have $[\AlmostBasicBC<\kappa>[\Lang](X)]^{V^+} = [\AlmostBasicBC<\kappa>[\Lang](X)]^{V^-}$. We will therefore refer to this set simply as $\AlmostBasicBC<\kappa>[\Lang](X)$. 

We say that two $\kappa$-Borel codes $\gamma_0, \gamma_1$ are $\mu$-equivalent when
\linebreak
$\mu^{V}(\gamma_0) = \mu^V(\gamma_1)$
for all $V \in \cV$.

Suppose that $\gamma = \bigvee \Gamma$ is an almost basic $\kappa$-Borel code. Then there is an almost basic $\kappa$-Borel code $\gamma^*  = \bigvee \Gamma^*$
that is $\mu$-equivalent to $\gamma$ such that
for any distinct $\zeta_0, \zeta_1\in \Gamma^*$, we have $\BCvalue(\zeta_0 \And \zeta_1)^V = \emptyset$ for all $V \in \cV$. Therefore 
$\mu(\gamma) = \mu^V(\gamma)$
for all $V \in \cV$.

Let $I\subseteq [\BorelCodes<\kappa>[\Lang](X)]^{V_0}$ be the collection of $\gamma$ such that
\[
\mu(\gamma \And \zeta) = \mu^{V}(\gamma \And \zeta) 
\qquad
\text{and}
\qquad
\mu(\gamma \Or \zeta) = \mu^{V}(\gamma \Or \zeta)
\]
for any $\zeta \in \AlmostBasicBC<\kappa>[\Lang](X)$  and any  $V \in \cV$.

Note that there are functions $A, O \:\AlmostBasicBC<\kappa>[\Lang](X) \times \AlmostBasicBC<\kappa>[\Lang](X) \to \AlmostBasicBC<\kappa>[\Lang](X)$ such that $A(\gamma_0, \gamma_1)$ is $\mu$-equivalent to $\gamma_0 \And \gamma_1$ and $O(\gamma_0, \gamma_1)$ is $\mu$-equivalent to $\gamma_0 \Or \gamma_1$. Therefore $\AlmostBasicBC<\kappa>[\Lang](X) \subseteq I$. 

Also note that there is an $n\:\AlmostBasicBC<\kappa>[\Lang](X) \to \AlmostBasicBC<\kappa>[\Lang](X)$ such that 
$n(\zeta)$ is \linebreak$\mu$-equivalent to $\neg \zeta$
for all $\zeta \in \AlmostBasicBC<\kappa>[\Lang](X)$.
Suppose $\gamma \in I$ and $\zeta \in \AlmostBasicBC<\kappa>[\Lang](X)$. Then $\mu(\neg \gamma \Or \zeta) = 1 - \mu(\gamma \And \neg \zeta)$ and $\mu(\neg \gamma \And \zeta) = 1 - \mu(\gamma \Or \neg \zeta)$. Therefore $I$ is closed under negation. In particular for $\gamma \in I$, for $\zeta \in \AlmostBasicBC<\kappa>[\Lang](X)$ and for any $V \in \cV$, we have $\mu(\gamma \SymDiff \zeta) = \mu^{V}(\gamma \SymDiff \zeta)$. 

Suppose $V^* \in\cV$ is such that $V^* \,\models \, |X| = |\Lang| = |\kappa| = \w$. Then $V^* \models \Str<\kappa>[\Lang](X)$ is a standard Borel space. Hence for every $\gamma \in [\BorelCodes<\kappa>[\Lang](X)]^{V^*}$ and $\epsilon \in \Rationals^+$ there is an $\eta(\epsilon, \gamma) \in \AlmostBasicBC<\kappa>[\Lang](X)$ such that $\mu^{V^*}(\gamma \SymDiff \zeta) < \epsilon$. 

Now suppose $\gamma_0, \gamma_1 \in I$ and $\zeta \in \AlmostBasicBC<\kappa>[\Lang](X)$. For $\epsilon \in \Rationals^+$ we have \linebreak $\mu^{V^*}(\gamma_0 \SymDiff \eta(\epsilon, \gamma_0)) < \epsilon$ and $\mu^{V^*}(\gamma_1 \SymDiff \eta(\epsilon, \gamma_1)) < \epsilon$. 
Therefore, as $\gamma_0, \gamma_1 \in I$, we have $\mu^{V}(\gamma_0 \SymDiff \eta(\epsilon, \gamma_0)) < \epsilon$ and $\mu^{V}(\gamma_1 \SymDiff \eta(\epsilon, \gamma_1)) < \epsilon$ 
for any $V \in \cV$. Therefore for any $V \in \cV$, we have 
\[
\bigl|\mu^V(\gamma_0 \And \gamma_1 \And \zeta) - \mu(\eta(\epsilon, \gamma_0) \And \eta(\epsilon, \gamma_1) \And \zeta)\bigr| < 2 \epsilon 
.\]

Then for any $V^-, V^+ \in \cV$  and $\zeta \in \AlmostBasicBC<\kappa>[\Lang](X)$ we have 
\[\bigl|\mu^{V^-}(\gamma_0 \And \gamma_1 \And \zeta) - \mu^{V^+}(\gamma_0 \And \gamma_1 \And \zeta)\bigr| <  4 \epsilon.
\]
But, as $\epsilon$ was arbitrary, this implies $\mu^{V^-}(\gamma_0 \And \gamma_1 \And \zeta) = \mu^{V^+}(\gamma_0 \And \gamma_1 \And \zeta)$. \linebreak A similar argument shows that $\mu^{V^-}((\gamma_0 \And \gamma_1) \Or \zeta) = \mu^{V^+}((\gamma_0 \And \gamma_1) \Or \zeta)$. Therefore $\gamma_0\And \gamma_1 \in I$.  As $I$ is closed under negation and $\mu$-equivalence, this also shows that $\gamma_0 \Or \gamma_1 \in I$.

Now suppose $\gamma = \bigvee \Gamma$ with $\gamma \in V_0$ and $\Gamma \subseteq I$. For $V \in \cV$ we have $\mu^V(\gamma) = \sup\{\mu^V(\bigvee \Gamma_0) \st \Gamma_0  \in \PowersetFin(\Gamma)\}$. But by the above 
we know that $\mu^V(\bigvee \Gamma_0) = \mu^{V_0}(\bigvee \Gamma_0)$ for any $\Gamma_0 \in \PowersetFin(\Gamma)$. Hence $\mu(\bigvee \Gamma) = \mu^{V}(\bigvee \Gamma)$ and $\bigvee \Gamma \in I$. Therefore, $I = [\BorelCodes<\kappa>[\Lang](X)]^{V_0}$. 

The result then follows as 
$\mu(\gamma) = \mu^{V_1}(\gamma)$
for $\gamma \in I$.
\end{proof}

By \cref{Relativizations give same values}, for $V_0 \subseteq V_1$ 
we have
$[\RandPO(\mu)]^{V_0} = [\RandPO(\mu^{V_1})]^{V_1} \cap V_0$, and for $\gamma, \zeta \in [\RandPO(\mu)]^{V_0}$ we have $V_0 \models \gamma \preceq_\mu \zeta$ if and only if $V_1 \models \gamma \preceq_{\mu^{V_1}} \zeta$. 
We also have the following 
corollary.

\begin{corollary}
\label{We can approximate all Borel codes by almost basic Borel codes}
Suppose $\gamma \in \BorelCodes<\kappa>[\Lang](X)$ and $\mu$ is a probability $\kappa$-measure on $\Str<\kappa>[\Lang](X)$. Then for all $\epsilon \in \Rationals^+$ there is a $\zeta_\epsilon \in \AlmostBasicBC<\kappa>[\Lang](X)$ such that $\mu(\gamma \SymDiff\zeta_\epsilon) < \epsilon$. 
\end{corollary}
\begin{proof}
Let $V_1$ be such that $V_1 \,\models\, |X| = |\Lang| = |\kappa| = \w$. Then $[\Str<\kappa>[\Lang](X)]^{V_1}$ is a standard Borel space and hence a Radon space. So for all $\epsilon \in \Rationals^+$ there must be a $\zeta_\epsilon \in [\AlmostBasicBC<\kappa>[\Lang](X)]^{V_1}$ such that $\mu^{V_1}(\gamma \SymDiff \zeta_\epsilon) < \epsilon$. The result then follows from \cref{Relativizations give same values} and the fact that $\AlmostBasicBC<\kappa>[\Lang](X)$ is the same in any model of set theory.
\end{proof}

Note that  for
$V_0 \subseteq V_1$,
in general it need not be the case that if 
\linebreak
$\gamma \in [\BorelCodes<\kappa>[\Lang](X)]^{V_0}$ and $V_0 \models \BCvalue(\gamma) = \emptyset$ then $V_1 \models \BCvalue(\gamma) = \emptyset$, as the following example shows.
\begin{example}
Let $\Lang = \{U, \preccurlyeq\} \cup \{f\}$ where $U$ is a unary relation symbol, $\preccurlyeq$~is a binary relation symbol,
and $f$ is a unary function symbol. For $\alpha, \beta \in \kappa^+$, let $T_{\alpha, \beta}$ be the sentence of $\Lkpw(\Lang)$ which asserts the conjunction of the following statements. 
\begin{itemize}
\item $\preccurlyeq$ is a linear ordering 
of order type $\alpha$
on the elements that satisfy $U$.

\item $\preccurlyeq$ is a linear ordering 
of order type $\beta$
on the elements that satisfy $\neg U$.

\item 
$U(x) \leftrightarrow \neg U(f(x))$ for all $x$.

\item 
$f(f(x)) = x$
for all $x$.
\end{itemize}
Note that $T_{\alpha, \beta}$ implies that $f$ is a bijection from the set of elements that satisfy $U$ to the set of elements that satisfy $\neg U$. Hence $T_{\alpha, \beta}$ has a model precisely when $|\alpha| =|\beta|$. 

Hence for $V_0 \subseteq V_1$, if
$V_0 \models \kappa > \w$ and $V_1 \models |\kappa|^{V_0} = \w$ then 
\[
V_0 \models \BCvalue(\BCformula(\bigwedge T_{\w, {\kappa}})) = \emptyset
\]
whereas 
\[
V_1 \models \BCvalue(\BCformula(\bigwedge T_{\w, {\kappa}})) \neq \emptyset.
\]
\end{example}

The next lemma essentially says that emptiness of almost basic
$\kappa$-Borel codes is absolute.

\begin{lemma}
\label{Almost basic Borel codes being empty is absolute}
Suppose $V_0 \subseteq V_1$ and $\gamma \in [\BorelCodes<\kappa>[\Lang](X)]^{V_0}$ is a finitary Boolean combination of basic codes.
Then $V_0 \models \BCvalue(\gamma) = \emptyset$ if and only if $V_1 \models \BCvalue(\gamma) = \emptyset$. 
\end{lemma}
\begin{proof}
This is because there is a first-order quantifier-free formula $\psi$ of arity $k$ and elements $x_0, \dots, x_{k-1}$ such that 
$V\models \BCvalue(\gamma) = \extent<\Lang>[X]{\psi(x_0,\dots, x_{k-1})}$
for any model of set theory $V$.
\end{proof}

This motivates the following definition. 
\begin{definition}
Suppose $\gamma_0, \gamma_1$ are finitary Boolean combinations of basic $\kappa$-Borel codes. We say that $\gamma_0 \sqsubseteq \gamma_1$ if $\BCvalue(\gamma_0 \And \neg \gamma_1) = \emptyset$. We say that $\gamma_0, \gamma_1$ are \defn{equivalent} if $\gamma_0 \sqsubseteq \gamma_1$ and $\gamma_1 \sqsubseteq \gamma_0$. 
\end{definition}

Note that by \cref{Almost basic Borel codes being empty is absolute}, 
equivalence of almost basic Borel codes is absolute.

\begin{theorem} 
\label{Basic results about relativization}
Suppose $V_0 \subseteq V_1$ and
$X, Y \in V_0 $. Let 
$\mu\in V_0$ be a probability $\kappa$-measure on $\Str<\kappa>[\Lang](X)$. 
Then the following hold.
\begin{itemize}
\item[(a)] There is a unique  probability $\kappa$-measure
$\mu^{V_1} \in V_1$ that is the relativization of $\mu$.

\item[(b)] $\mu$ is invariant if and only if $\mu^{V_1}$ is invariant.

\item[(c)] If $\eta \in [\BorelCodes<\kappa>[\Lang](X)]^{V_0}$  then $\mu(\eta) = \mu^{V_1}(\eta)$.

\item[(d)] If $\mu$ is invariant then $[\mu_Y]^{V_1} = (\mu^{V_1})_{Y}$. 

\item[(e)] If $\mu$ is invariant and $i\:X \to Y$ is injective then 
$\mu(\gamma) = \mu_Y(i(\gamma))$
for any $\gamma \in [\BorelCodes<\kappa>[\Lang](X)]^{V_0}$.
\end{itemize}
\end{theorem}

\begin{proof} 
We have (a) by \cref{Unique-extension}, and (b) by \cref{Finite support invariance gives full invariance}. We have (c) by \cref{Relativizations give same values}, and (d) by Propositions \ref{Unique-extension} and \ref{Invariant-measures-lift}. Finally, (e) follows from a straightforward induction on the complexity of Borel codes. 
\end{proof}

Note that for $\varphi \in \Lkpw(\Lang)$, if $X$ and $Y$ are each of size at most $\kappa$, then
the $\kappa$-Borel code in $\BorelCodes<\kappa>[\Lang](X)$ which defines $\extent<\Lang>[X]{\varphi}$ is in general different from the $\kappa$-Borel code in $\BorelCodes<\kappa>[\Lang](Y)$ that defines $\extent<\Lang>[Y]{\varphi}$, because quantifiers are converted to conjunctions and disjunctions.
However, the next result shows that the almost-sure theory of an invariant measure $\mu$ on $\Str<\kappa>[\Lang](X)$ is the same as the almost-sure theory of $\mu_Y$. Further, this theory is independent of the model of set theory in which one is working.

\begin{theorem}
\label{Theory of an invariant is absolute}
Suppose $V_0 \subseteq V_1$.
Let $\varphi \in [\Lkpw(\Lang)]^{V_0}$ be a sentence. 
Suppose 
$X, Y \in V_0$ with $V_0 \models \max\{|X|, |Y|\} \leq \kappa$. 
Then for any invariant measure $\mu$ on $\Str<\kappa>[\Lang](X)$, we have
\[
\mu(\varphi) = 1 
\qquadiff 
\mu_Y(\varphi) = 1,
\]
i.e.,
\[
[\Th[\kappa](\mu)]^{V_0}=[\Th[\kappa]([\mu_Y]^{V_1})]^{V_1} \cap V_0.
\]
\end{theorem}
\begin{proof}
First note that for any $\varphi \in [\Lkpw(\Lang)]^{V_0}$, the equivalence between \linebreak $\varphi \in [\Th[\kappa](\mu)]^{V_0}$ and $\varphi \in [\Th[\kappa](\mu^{V_1})]^{V_1}$ follows from \cref{Exists Borel codes for formulas} and \cref{Relativizations give same values}. 

But if $|X| = |Y|$ then $\Th[\kappa](\mu) = \Th[\kappa](\mu_Y)$. 
Therefore if $V_1 \subseteq V^*$ and $V^* \models |X| = |Y|$, we have 
\begin{align*}
[\Th[\kappa](\mu)]^{V_0} 
&= [\Th[\kappa](\mu^{V^*})]^{V^*} \cap V_0 \\
&= [\Th[\kappa]((\mu^{V^*})_Y)]^{V^*} \cap V_0 \\
&= [\Th[\kappa]([\mu_Y]^{V^*})]^{V^*} \cap V_0 \\
&= [\Th[\kappa](\mu_Y)]^{V_0}.
\end{align*}
The result then follows. 
\end{proof}

If $\Lang$ is countable and $\cM$ is a countably infinite $\Lang$-structure then there is a sentence $\tau_{\cM} \in \Lwow(\Lang)$, called a Scott sentence of $\cM$, such that a countably infinite 
$\Lang$-structure
satisfies $\tau_{\cM}$ if and only if it is isomorphic to $\cM$. Therefore for any countably infinite $X_0 \subseteq X$, if $\mu_{X_0}$ is concentrated on the isomorphism class of $\cM$, then
for any countably infinite $X_1 \subseteq X$ the measure $\mu_{X_1}$ is concentrated on the isomorphism class of $\cM$,
by \cref{Theory of an invariant is absolute}.

\begin{remark}
Suppose $\Lang$ contains a relation symbol
$R$ of arity $n+1$.  There is a sentence $\psi_R^\mathrm{f}$ which holds in any $\Lang$-structure $\cM$
if and only if $R^\cM$ is the graph of a function with $n$ inputs. There is also a sentence $\psi_R^\mathrm{cf}$ which holds in any $\Lang$-structure $\cM$ if and only if $R^\cM$ is the graph of a choice function with $n$ inputs. 

By \cite{complete-classification},
for any relation symbol $R \in \Lang$ and any invariant measure $\mu$, we have $(\psi_R^\mathrm{f} \rightarrow \psi_R^\mathrm{cf}) \in \Th[\kappa](\mu)$. In other words, 
$\mu$ assigns measure 1 to the set of structures in which any function defined by a relation must be a choice function.
Therefore, relaxing the assumption that our language is relational 
would not add much
generality 
to
the study of invariant measures. 
\end{remark}

\subsection{Dissociated, Ergodic and Extreme Invariant Measures} 
\label{Section: Ergodic Extreme dissociated}

In this subsection we show that for any invariant measure, being dissociated, ergodic, or extreme are all equivalent notions.
We do this by showing that each of these notions is absolute and then using \cref{If countable dissociated equals ergodic equals extreme}.

\begin{proposition}
\label{Dissociated is absolute}
Suppose $V_0 \subseteq V_1$, and let $\mu\in V_0$ be an invariant measure on $\Str<\kappa>[\Lang](X)$. 
Then $\mu$ is dissociated in $V_0$ if and only if $\mu^{V_1}$ is dissociated in $V_1$. 
\end{proposition}
\begin{proof}
First note that if $\mu^{V_1}$ is dissociated 
then $\mu$ is dissociated 
by \cref{Relativizations give same values}. 

Now suppose $\mu$ is dissociated.
Let $X_0, X_1 \subseteq X$ be disjoint finite sets and let $A_0 \in \BorelCodes<\kappa>[\Lang](X_0)$ and 
$A_1 \in \BorelCodes<\kappa>[\Lang](X_1)$ be in $V_1$. By \cref{We can approximate all Borel codes by almost basic Borel codes}, for each $\epsilon \in \Rationals^+$ there 
exist 
$\gamma^0_\epsilon \in \AlmostBasicBC<\kappa>[\Lang](X_0)$ and $\gamma^1_\epsilon \in \AlmostBasicBC<\kappa>[\Lang](X_1)$
such that $\mu^{V_1}(A_i \SymDiff \gamma_\epsilon^i) < \epsilon$ for $i \in \{0, 1\}$. Hence 
\[
|\mu^{V_1}(A_0 \cap A_1) -  \mu^{V_1}(\gamma_\epsilon^0 \And \gamma_\epsilon^1)| \leq 2 \epsilon
\]
and 
\[|\mu^{V_1}(A_0) \cdot \mu^{V_1}(A_1) -  \mu^{V_1}(\gamma_\epsilon^0) \cdot \mu^{V_1}(\gamma_\epsilon^1)| < 3 \epsilon 
.
\]
But $\mu$ is dissociated
and so \[
|\mu^{V_1}(\gamma_\epsilon^0 \And \gamma_\epsilon^1) -\mu^{V_1}(\gamma_\epsilon^0) \cdot \mu^{V_1}(\gamma_\epsilon^1)| = 0.
\]
Therefore we have 
\[|\mu^{V_1}(A_0 \cap A_1) -  \mu^{V_1}(A_0) \cdot \mu^{V_1}(A_1)| <  5 \epsilon.
\]
Because $\epsilon$ was arbitrary, this implies that $\mu^{V_1}(A_0 \cap A_1) =  \mu^{V_1}(A_0) \cdot \mu^{V_1}(A_1)$. Hence $\mu^{V_1}$ is dissociated.
\end{proof}

We now 
show that 
extremality
of 
an invariant measure
is absolute. 
We will first need a combinatorial lemma. 

\begin{lemma}
\label{Finite presheaf compactness}
Let $(\sC, \sCle)$
be a directed partial order (i.e., a poset in which every finite collection of elements has an upper bound),
and let $F$ be a presheaf on $(\sC, \sCle)$.
For $A \sCle B$, write
$i_{A, B}$ for the corresponding map in $(\sC, \sCle)$ when treated as a category.
For $A \sCle B$ and $b \in F(B)$, write $b\rest[A]$ for $F(i_{A, B})(b)$.

Suppose that 
$0<|F(A)| < \w$ 
for all $A \in \sC$.
Then there is a sequence $\<a_A\>_{A \in \sC}$ 
such that
\begin{itemize}
\item 
$a_A \in F(A)$
for all $A \in \sC$ and

\item 
$a_A = a_B \rest[A]$
whenever $A \sCle B \in \sC$. 
\end{itemize}
\end{lemma}
\begin{proof}
Let $\Lang$ be the language containing
\begin{itemize}
\item a sort $S_A$ for each $A \in \sC$,

\item 
a constant symbol $c_a$ of sort $S_A$  for each $a \in F(A)$, 

\item 
a function symbol $f_{A, B}$ from $S_B$ to $S_A$
for each $A \sCle B$, and 

\item 
a constant symbol $d_A$
for each $A \in \sC$.
\end{itemize}
Let $T$ be the theory which says,
for each $A \sCle B \in \sC$, that 
\begin{itemize}
\item 
$(\forall x:S_A)\bigvee_{a \in F(A)} (x = c_a)$,

\item $f_{A, B}(c_b) = c_{b\,\rest[A]}$, and

\item $d_A = f_{A, B}(d_B)$.
\end{itemize}

Suppose $T_0$ is a finite subset of $T$.
Note that $T_0$ mentions $d_A$ for at most finitely many elements $A_0, \ldots, A_{n-1} \in \sC$.
As $\sC$ is directed, there must be some $B$ such that $A_i \sCle B$ for all $i \in [n]$.
Further, as $F(B) \neq \emptyset$, there is some $b \in F(B)$.
Let $\cM_0$ be an $\Lang$-structure with underlying set \linebreak $\{(A, a) \st a \in F(A)\}$ where $c_a^{\cM_0} = (A, a)$ for $a \in F(A)$. 
Let 
$d_{A_i}^{\cM_0} = c_{b\,\rest[A_i]}^{\cM_0}$
for $i \in [n]$.
For $B \in \sC \setminus \{A_i\}_{i \in [n]}$, let $d_B^{\cM_0}$ be an arbitrary element of $M_0$ of the form $(B, b)$, which must exist as $F(B) \neq \emptyset$. 
Then $\cM_0$ clearly satisfies $T_0$. As $T_0$ was arbitrary, $T$ is first-order and finitely satisfiable, and hence has a model $\cM$.
But then for all $A \in \sC$ there is an $a_A$ such that $\cM \models d_{A} = c_{a_A}$. 
Hence $\<a_A\>_{A \in \sC}$
is the desired collection. 
\end{proof}

\begin{proposition}
\label{Extreme is absolute}
Suppose $V_0 \subseteq V_1$, and let $\mu\in V_0$ be an invariant measure on $\Str<\kappa>[\Lang](X)$. Then $\mu$ is extreme in $V_0$ if and only if $\mu^{V_1}$ is extreme in $V_1$. 
\end{proposition}
\begin{proof}
Let $\sC$ be the collection of triples $(n, \Lang_0, X_0)$
where $n \in \w$, where \linebreak $\Lang_0 \subseteq \Lang$ is finite, and where $X_0 \subseteq X$ is finite. Let $(n_0, \Lang_0, X_0) \sCle (n_1, \Lang_1, X_1)$ if and only if $n_0 \leq n_1$ and $\Lang_0 \subseteq \Lang_1$ and $X_0 \subseteq X_1$. Note that $(\sC, \sCle)$ is directed. 

Define an $\epsilon$-approximate measure on $\Str<\kappa>[\Lang_0](X_0)$ to be a map 
\linebreak
$\nu\:\BasicBC<\kappa>[\Lang_0](X_0) \to [0, 1]$ such that
\begin{itemize}
\item $\nu(\BCformula(\bigwedge \emptyset)) \geq 1 - \epsilon$, 

\item if $\varphi$ and $\psi$ are finite conjunctions of $\Lang_0$-literals with $\models \varphi(\xx) \rightarrow \psi(\xx)$, then 
$\nu^\circ(\BCformula(\psi(\aa))) + \epsilon \geq \nu^\circ(\BCformula(\varphi(\aa))) - \epsilon$
for every $\aa \in X_0$, 

\item 
$\nu(\psi(\aa) \And \neg \psi(\aa)) \leq \epsilon$
for any $\Lang_0$-literal $\psi(\xx)$ and $\aa \in X_0$,   and

\item for all tuples $\aa \in X_0$ and 
sequences $\<\psi_i(\xx)\>_{i \in [n+1]}$
of $\Lang_0$-literals, we have 
\[
\bigg|\nu^\circ\bigg(\bigwedge_{i \in [n]} \psi_i(\aa)\bigg) - \mu^\circ\bigg(\psi_n(\aa) \And \bigwedge_{i \in [n]} \psi_i(\aa)\bigg) + \mu^\circ\bigg(\neg \psi_n(\aa) \And \bigwedge_{i \in [n]} \psi_i(\aa)\bigg)\bigg| \leq \epsilon.
\]
\end{itemize}

For a measure $\nu$ on $\Str<\kappa>[\Lang_0](X_0)$ and $n \in \w$ we let 
\[
a(\nu, n)(\gamma) = \frac{\floor{2^{n+1} \cdot \nu(\gamma)}}{2^{n+1}}.
\]
Note that $a(\nu, n)$ is a $2^{-n}$-approximate measure.

For $q \in \Rationals^+$ let $F_q$ be the presheaf on $\sC$ where
$F_q(n,\Lang_0, X_0)$ is the collection of tuples $(\nu, \eta, r, s)$ such that 
\begin{itemize}
\item $\nu$ and $\eta$ are $2^{-n}$-approximate measures, 

\item $r, s \in \{\frac{k}{2^n}\}_{0 < k < 2^n}$
where
$r + s = 1$  and
$q < r, s$, 

\item for all $\gamma \in \AlmostBasicBC<\kappa>[\Lang_0](X_0)$, 
\begin{itemize}
\item $\nu(\gamma), \eta(\gamma) \in \{\frac{k}{2^n}\}_{0 \leq k \leq 2^n}$ and

\item $|\mu(\gamma) - (r\cdot \nu(\gamma)+ s \cdot \eta(\gamma))| < 2^{-n+2}$.
\end{itemize}
\end{itemize}

If $(n_0, \Lang_0, X_0) \sCle (n_1, \Lang_1, X_1)$ and $(\nu, \eta, r, s) \in F(n_1, \Lang_1, X_1)$ let \linebreak $(\nu, \eta, r, s)\rest[(n_0, \Lang_0, X_0)]$ be the unique element $(\nu_0, \eta_0, \frac{\floor{2^{n_0} \cdot r}}{2^{n_0}}, \frac{\floor{2^{n_0} \cdot s}}{2^{n_0}})$ for which $\nu_0(\gamma) = \frac{\floor{2^{n_0} \cdot \nu(\gamma)}}{2^{n_0}}$ and $\eta_0(\gamma) = \frac{\floor{2^{n_0} \cdot \eta(\gamma)}}{2^{n_0}}$ for all $\gamma \in \AlmostBasicBC<\kappa>[\Lang_0](X_0)$. 

Suppose that $\mu$ is not extreme, i.e., there are distinct invariant 
measures $\nu, \eta$ and
$r, s \in (0, 1)$ such that $\mu = r \cdot \nu + s \cdot \eta$. Let $q \in \Rationals^+$ be such that $q < r, s$. Then for any $(n_0, X_0, \Lang_0) \in \sC$, we have 
\[\Bigl(a(\nu, {n_0}), a(\eta, n_0), \frac{\floor{2^{n_0} \cdot r}}{2^{n_0}}, \frac{\floor{2^{n_0} \cdot s}}{2^{n_0}}\Bigr) \in F(n_0, X_0, \Lang_0),
\]
and hence $F_q(n_0, X_0, \Lang_0)$ is not empty.

Now suppose that $q \in \Rationals^+$ and  $F_q(n_0, X_0, \Lang_0) \neq \emptyset$ for every \linebreak 
$(n_0, X_0, \Lang_0) \in \sC$. By \cref{Finite presheaf compactness} there is a function $g$ with domain $\sC$ where $g(A) \in F_q(A)$ for $A \in \sC$ and where $g(B)\rest[A] = g(A)$ for $A \sCle B \in \sC$. If $\gamma \in \BasicBC<\kappa>[\Lang](X)$ then there must be an $\Lang_0 \in \PowersetFin(\Lang)$ and $X_0 \in \PowersetFin(X)$ such that $\gamma \in \BasicBC<\kappa>[\Lang_0](X_0)$. Define
\begin{itemize}
\item $\nu^\circ(\gamma) = \lim_{n \to \infty} \pi_0(g(n, \Lang_0, X_0))(\gamma)$, 

\item $\eta^\circ(\gamma) = \lim_{n \to \infty} \pi_1(g(n, \Lang_0, X_0))(\gamma)$, 

\item $r = \lim_{n \to \infty}  \pi_2(g(n, \Lang_0, X_0))$,  and 

\item $s =  \lim_{n \to \infty} \pi_3(g(n, \Lang_0, X_0))$,
\end{itemize}
where $\pi_i$ is the projection to coordinate $i$ of its $4$-tuple input.
Such limits always exist by the definition of $F_q$.  Further, because of how
$\nu^\circ$ and $\eta^\circ$ were defined, by \cref{Unique-extension} there are unique 
probability $\kappa$-measures $\nu, \eta$ on $\Str<\kappa>[\Lang](X)$ such that $\nu$ agrees with $\nu^\circ$ on $\BasicBC<\kappa>[\Lang](X)$ and $\eta$ agrees with $\eta^\circ$ on $\BasicBC<\kappa>[\Lang](X)$. But also, for any $\gamma \in \BasicBC<\kappa>[\Lang](X)$, we have 
\[
\mu(\gamma) = r \cdot \nu^\circ(\gamma) + s\cdot \eta^\circ(\gamma) = r \cdot \nu(\gamma) + s\cdot \eta(\gamma). 
\]
Hence 
$\mu = r \cdot \nu + s \cdot \eta$
by \cref{Unique-extension}.

Therefore $\mu$ is not extreme if and only if there is some $q \in \Rationals^+$ such that for all $A \in \sC$ we have  $F_q(A) \neq \emptyset$.  But in fact, for each $q \in \Rationals^+$, we have $[F_q]^{V_0}  = [F_q]^{V_1}$. Hence $\mu$ is extreme in $V_0$ if and only if it is extreme in $V_1$. 
\end{proof}

We now characterize
$\mu$-almost invariant sets 
in terms of Borel codes. 

\begin{lemma}
\label{Almost invariant on finite permutations is full almost invariant}
Let $\mu$ be an invariant 
measure
on $\Str<\kappa>[\Lang](X)$ and suppose $\gamma \in \BorelCodes<\kappa>[\Lang](X)$ satisfies
$\mu(h[\gamma] \SymDiff \gamma) = 0$
for all $h \in \SymFin(X)$.
Then $\gamma$ is $\mu$-almost invariant. 
\end{lemma}
\begin{proof}
Let $g \in \Sym(X)$. For $\epsilon > 0$, by \cref{We can approximate all Borel codes by almost basic Borel codes} there is some $\zeta \in \AlmostBasicBC<\kappa>[\Lang](X)$ such that $\mu(\gamma \SymDiff \zeta) < \epsilon$. As $\zeta \in \AlmostBasicBC<\kappa>[\Lang](X)$, there is an $h \in \SymFin(X)$ such that $h[\zeta] = g[\zeta]$. Therefore 
\begin{align*}
\mu(\gamma \SymDiff g[\gamma]) &\leq \mu(\gamma \SymDiff h[\gamma]) +  \mu(h[\gamma] \SymDiff h[\zeta]) + \mu(h[\zeta] \SymDiff g[\zeta]) + \mu(g[\zeta] \SymDiff g[\gamma]) \\
&= 0 + \epsilon + 0 + \epsilon \\
&= 2 \epsilon.
\end{align*}
As $\epsilon$ was arbitrary, $\gamma$ is $\mu$-almost invariant.
\end{proof}

\begin{proposition}
\label{Ergodic is absolute}
Suppose $V_0 \subseteq V_1$ and suppose $\mu\in V_0$ is an invariant measure on $\Str<\kappa>[\Lang](X)$. Then $\mu$ is ergodic in $V_0$ if and only if $\mu^{V_1}$ is ergodic in $V_1$. 
\end{proposition}
\begin{proof}
By 
\cref{Relativizations give same values}
and \cref{Almost invariant on finite permutations is full almost invariant}, 
a given $\gamma \in [\BorelCodes<\kappa>[\Lang](X)]^{V_0}$ is $\mu$-almost invariant (in $V_0$) if and only if it is $\mu^{V_1}$-almost invariant (in $V_1$). Therefore, because $[\BorelCodes<\kappa>[\Lang](X)]^{V_0} \subseteq [\BorelCodes<\kappa>[\Lang](X)]^{V_1}$, if $\mu^{V_1}$ is ergodic (in $V_1$) then $\mu$ is ergodic (in $V_0$).

\begin{claim}
If $\mu^{V_1}$ is non-ergodic (in $V_1$) then $\mu_\w$ is non-ergodic (in $V_0$).
\end{claim}
\begin{proof}
Suppose $\mu^{V_1}$ is non-ergodic, i.e., there is a $\gamma\in [\BorelCodes<\kappa>[\Lang](X)]^{V_1}$ such that 
$\mu^{V_1}(\gamma) \not \in \{0, 1\}$
and
$\mu^{V_1}(g[\gamma] \SymDiff \gamma) = 0$ 
for all  $g \in [\SymFin(X)]^{V_1}$.

Let $V^*$ be a model of set theory where $V_1 \subseteq V^*$ and where \linebreak $V^* \models |X| = |\Lang| =  \w$. Note that $[\SymFin(X)]^{V_1} = [\SymFin(X)]^{V^*}$, and so by \cref{Almost invariant on finite permutations is full almost invariant}, the code $\gamma$ witnesses that $\mu^{V^*}$ is not ergodic. But as $V^* \models |X| = \w$, this implies that $(\mu^{V^*})_\w$ 
is not ergodic. Hence 
$[\mu_\w]^{V^*}$ is not ergodic by \cref{Basic results about relativization}~(d).
Let $q \in \Rationals^+$ be such that $q < [\mu_\w]^{V_*}(\gamma) < 1 - q$. 

We now define a tree $(T, \preceq)$ where level $n$ of $T$ consists of pairs $(n, \zeta)$
such that 
there is a
$\gamma \in \AlmostBasicBC<\kappa>[\Lang](\w)$ for which 
\begin{itemize}
\item $q - 2^{-n} < \mu_\w(\gamma) < 1-q + 2^{-n}$, and

\item 
$\mu_\w(\zeta \SymDiff g[\zeta]) < 2^{-n}$
for all $g \in \SymFin(\w)$.
\end{itemize}
We let 
$(m,\zeta_m) \prec_1 (n, \zeta_n)$ if and only if $n = m+1$ and $\mu(\zeta_m \SymDiff \zeta_n) \preceq 2^{-m}$. We let $\preceq$ be the transitive closure of $\prec_1$ (along with equality).

Note that $(T, \preceq)$ is the same whether  it is defined in $V_0$ or $V^*$. Also note that 
$(T, \preceq)$ is ill-founded 
in $V^*$, 
and so $(T, \preceq)$ is ill-founded in $V_0$. 

Suppose $\<(k, \zeta_k)\>_{k \in \w}$ is an infinite branch of $(T, \preceq)$ in $V_0$. Let \linebreak $\eta = \bigwedge_{n \in \w} \bigvee_{k \geq n} \zeta_k$. 
Note that for any $n \in \w$ and $g \in \SymFin(X)$ we have 
\[
\mu_\w\bigg(\bigvee_{k \geq n} \zeta_k \SymDiff g\bigg[\bigvee_{k \geq n} \zeta_k\bigg]\bigg) \leq \sum_{k \geq n} \mu_\w\bigg(\zeta_k\SymDiff g[\zeta_k]\bigg) \leq \sum_{k \geq n}2^{-k} = 2^{-n+1}.
\]
Therefore $\mu_\w(\eta \SymDiff g[\eta]) = 0$. 

Let $n \in \w$ be such that $2^{-n+1} < q$. We then have 
\[
\mu_\w(\zeta \SymDiff \gamma_n) \leq \sum_{k \geq n} \mu_\w(\gamma_n \SymDiff \gamma_k) \leq \sum_{k \geq n} 2^{-k} = 2^{-n+1}.
\]
Hence $0 < \mu_\w(\zeta) < 1$ and $\zeta$ witnesses that $\mu_\w$ is not ergodic, completing the proof of the claim.
\end{proof}

Now for an infinite set $X_0$ and an injection $i\:X_0 \to X$, if $\mu$ is ergodic then $\mu_{X_0}$ is ergodic as well by \cref{Basic results about relativization} (e)
and the fact that \linebreak $\BorelCodes<\kappa>[\Lang](X_0) \subseteq \BorelCodes<\kappa>[\Lang](X)$. 

In particular, if $\mu_\w$ is not ergodic then $\mu$ is not ergodic either. Hence if $\mu^{V_1}$ is not 
ergodic then $\mu$ is not ergodic.
Therefore ergodicity is absolute, completing the proof of the proposition.
\end{proof}

One of the reasons we are interested in ergodic invariant measures is 
that they give rise to complete theories, as we show in the next lemma.

\begin{lemma}
Let $\varphi \in \Lkpw(\Lang)$ be a sentence
and suppose $|X| \leq \kappa$.
Let $\mu$ be an ergodic invariant 
measure on $\Str<\kappa>[\Lang](X)$. Then either $\varphi \in \Th[\kappa](\mu)$ or $\neg \varphi \in \Th[\kappa](\mu)$. 
\end{lemma}
\begin{proof}
This follows immediately from the fact that for any $g \in \Sym(X)$, the codes $g[\BCformula({\varphi(\xx)})]$ and $\BCformula({\varphi(\xx)})$ interpret the same set. 
\end{proof}

In other words, every ergodic invariant measure has a complete almost-sure theory in $\Lkpw(\Lang)$.

Putting the above results together,
we have the following equivalences with ergodicity.
\begin{theorem}
\label{Equivalences with ergodic}
Suppose  $\mu$ is an invariant measure on $\Str<\kappa>[\Lang](X)$ with \linebreak $\Lang, X, \mu \in V$. Then the following are equivalent. 
\begin{itemize}
\item[(a)] $V \models \mu\text{ is ergodic}$. 

\item[(b)] 
$V^* \models \mu^{V^*}\text{ is ergodic}$,
for some model of set theory $V^*$ with $V \subseteq V^*$.

\item[(c)] 
$V^* \models \mu^{V^*}\text{ is ergodic}$,
for any model of set theory $V^*$ with $V \subseteq V^*$.

\item[(d)] 
$V \models \mu_Y\text{ is ergodic}$,
for any infinite set $Y \in V$.

\item[(e)] $V \models \mu\text{ is extreme}$. 

\item[(f)] $V \models \mu\text{ is dissociated}$. 

\end{itemize}
\end{theorem}
\begin{proof}
The implication from (c) to (b) is immediate. The implications from (b) to (a) and from (a) to (c) both follow from \cref{Ergodic is absolute}. 

Let $V^*$ be a model of set theory where $V \subseteq V^*$ and where \linebreak $V^* \models |\Lang| = |X| = |Y| = \w$.
We then have 
\[
V^* \models \mu\text{ is ergodic if and only if } \mu_Y\text{ is ergodic}.
\]
Therefore, by the equivalence of (a) and (b) we have that (a) and (d) are equivalent as well. 

By \cref{If countable dissociated equals ergodic equals extreme},
we also have that
\[
V^* \models \mu\text{ is ergodic if and only if }\mu \text{ is extreme}
\]
and 
\[
V^* \models \mu\text{ is ergodic if and only if }\mu \text{ is dissociated}.
\]
Therefore by 
Propositions~\ref{Dissociated is absolute} and \ref{Extreme is absolute}, we have
that (a) is equivalent to (e) and that (a) is equivalent to (f). 
\end{proof}

\section{Random Generic Structures}
\label{Random Generic Structures Section}

In this section we consider forcing with an invariant measure. 

\begin{definition}
Let $\mu$ be a probability $\kappa$-measure on $\Str<\kappa>[\Lang](X)$. Define a \linebreak \defn{$\mu$-random generic filter} over $V$ to be a filter on $\RandPO(\mu)$ that is generic over $V$. 
\end{definition}

Note that if $|X|\leq \kappa$ and $|\Lang| \leq \kappa$ then for any $\zeta \in \BorelCodes<\kappa>[\Lang](X)$ we have $\bigcap_{\zeta \sim_\mu \zeta^*} \BCvalue(\zeta^*) = \emptyset$, as for every $\cM \in \Str<\kappa>[\Lang](X)$ there is some $\gamma_\cM \in \BorelCodes<\kappa>[\Lang](X)$ such that $\BCvalue(\gamma_{\cM}) = \{\cM\}$ and $(\zeta \And \neg \gamma_\cM) \sim_\mu \zeta$. 

However, if $V_0 \subseteq V_1$, then we need not have 
\[ V_1 \models \bigcap \{\BCvalue(\gamma^*) \st \gamma^* \in [\BorelCodes<\kappa>[\Lang](X)]^{V_0}\text{ and } \gamma \sim_\mu \gamma^*\} = \emptyset.
\]
In fact, the following theorem shows that the intersection of the interpretations of elements of a $\mu$-random generic filter, which is necessarily closed under $\sim_\mu$, always contains a unique element. 

\begin{theorem}
\label{mu-generic filters are determined by their intersection}
Let $\mu \in V$ be a probability $\kappa$-measure on $[\Str<\kappa>[\Lang](X)]^V$, and suppose $\Generic$ is a $\mu$-random generic filter over $V$. Let $\Generic_b = \Generic \cap \BasicBC<\kappa>[\Lang](X)$. 
Then the following hold.
\begin{itemize}
\item[(a)] $V[\Generic] \models |\bigcap_{g \in \Generic} \BCvalue(g)|  =  1$. 

\item[(b)] $V[\Generic] \models \bigcap_{g \in \Generic} \BCvalue(g) = \bigcap_{g \in \Generic_b} \BCvalue(g)$. 

\item[(c)] If \, $V[\Generic]\models \{\cM_{\Generic}\} = \bigcap_{g \in \Generic_b} \BCvalue(g)$ \, then 
\[
\Generic = \{\zeta  \in [\RandPO(\mu)]^{V}\st V[\Generic] \models \cM_{\Generic} \in \BCvalue(\zeta)\}.
\]
\end{itemize}

\end{theorem}
\begin{proof} 
Note that (b) follows from (a) and (c). We first prove (a). Suppose $\Lang_0 \in \PowersetFin(\Lang)$ and $X_0 \in \PowersetFin(X)$, and let $\aa$ be an enumeration of $X_0$. For $\cM \in \Str[\Lang_0](X_0)$ let $\tau_{\cM}(\xx)$ be a first-order quantifier-free formula of arity $|X_0|$ such that 
\begin{itemize}
\item  $\cM \models \tau_{\cM}(\aa)$, and

\item if $\bb$ is an enumeration of $X_0$ and $\cN \in \Str[\Lang_0](X_0)$ where $\cN \models \tau_{\cM}(\bb)$ then the map $\aa \mapsto \bb$ is an isomorphism from $\cM$ to $\cN$. 
\end{itemize}
Note that we can always find such a $\tau_{\cM}$ as both $X_0$ and $\Lang_0$ are finite. Let
\[
P(\cM) = \{\gamma \in [\RandPO(\mu)]^{V} \st \gamma \preceq_\mu \BCformula(\tau_{\cM}(\aa))\}.
\]     
We then let 
\[
P(X_0, \Lang_0) = \bigcup_{\cM \in \Str[\Lang_0](X_0)} P(\cM).
\]
Note that $P(X_0,\Lang_0)$ is dense in $\RandPO(\mu)$. Therefore there must be some $Z(X_0, \Lang_0) \in \Generic$ and $\cM_{X_0, \Lang_0} \in \Str[\Lang_0](X_0)$ with $Z(X_0, \Lang_0) \in P(\cM_{X_0, \Lang_0})$. Further, if $\cM^* \in \Str[\Lang_0](X_0)$ and $Z'(X_0, \Lang_0) \in \Generic \cap P(\cM^*)$ then we must have $\cM^* = \cM_{X_0, \Lang_0}$ as $\Generic$ is a filter and all elements are compatible. Therefore we have $|\bigcap_{g \in \Generic_b} \BCvalue(g)| \leq 1$.

If $X_0 \subseteq X_1 \subseteq X$ and $\Lang_0 \subseteq \Lang_1 \subseteq\Lang$ then $\cM_{X_1, \Lang_1}\rest[X_0,\Lang_0] = \cM_{X_0, \Lang_0}$. Hence there is an $\cM_{\Generic} \in [\Str[\Lang](X)]^{V[\Generic]}$ such that for any finite $X_0 \subseteq X$ and finite $\Lang_0 \subseteq \Lang$, we have $\cM_G\rest[X_0,\Lang_0] = \cM_{X_0, \Lang_0}$. Therefore $\{\cM_{\Generic}\} = \bigcap_{g \in \Generic_b} \BCvalue(g)$ and so (a) holds. 

We now prove (c). It 
suffices to show that for any $\gamma \in [\BorelCodes<\kappa>[\Lang](X)]^{V}$ 
we have
\begin{align*}
\label{mu-generic filters are determined by their intersection: Eq 1}
\tag{$\circ$}
\gamma \in \Generic &\quadiff V[\Generic] \models \cM_{\Generic} \in \BCvalue(\gamma).
\end{align*}
We prove \eqref{mu-generic filters are determined by their intersection: Eq 1} by induction on the rank of $\gamma$. 

First note that \eqref{mu-generic filters are determined by their intersection: Eq 1} holds 
when
$\rank(\gamma) = 0$, i.e., $\gamma \in \BasicBC<\kappa>[\Lang](X)$. 

Now 
assume $\rank(\gamma) = \alpha$
and suppose that \eqref{mu-generic filters are determined by their intersection: Eq 1} holds for all elements of $\BasicBC<\kappa>[\Lang](X)$ of rank $< \alpha$. 

If $\gamma = \neg \zeta$ then, as $\Generic$ is generic, we have precisely one of $\gamma \in \Generic$ or $\zeta \in \Generic$. Therefore $\gamma \in \Generic$ if and only if $\zeta \not \in \Generic$ if and only if $V[\Generic] \not\models \cM_{\Generic} \in \BCvalue(\zeta)$ if and only if $V[\Generic] \models \cM_{\Generic} \in \BCvalue(\neg \zeta)$. 

Suppose we can write $\gamma = \bigvee \Gamma$ where each element of $\Gamma$ has rank $< \alpha$.
Suppose $V[\Generic] \models \cM_{\Generic} \in \BCvalue(\gamma)$. Then $V[\Generic] \models \cM_{\Generic} \in \bigcup_{\zeta \in \Gamma} \BCvalue(\zeta)$ and so \linebreak $V[\Generic] \models (\exists \zeta \in \Gamma)\, (\cM_{\Generic} \in \BCvalue(\zeta))$.
But then  $V[\Generic] \models (\exists \zeta \in \Gamma)\, (\zeta \in \Generic)$,
which implies 
that 
$\gamma \in \Generic$. 

Towards the converse 
when $\gamma = \bigvee \Gamma$, now suppose $\gamma \in \Generic$.  Let 
\[
D_\gamma = \{\neg \gamma\} \cup \{\eta \in \RandPO(\mu) \st (\exists \zeta \in \Gamma)\, (\eta \preceq_\mu \zeta)\}.
\]
Let $\nu \in \BorelCodes<\kappa>[\Lang](X)$ with $\nu \not \preceq_\mu \neg \gamma$, i.e., 
with 
$\mu(\gamma \And \nu) > 0$. Then, as \linebreak $\mu(\gamma\And \nu) \leq \sum_{\zeta \in \Gamma} \mu(\zeta \And \nu)$, there must be some $\zeta \in \Gamma$ with $\mu(\zeta \And \nu) > 0$. Hence $(\zeta \And \nu) \in D_\gamma$ (and in particular $(\zeta \And \nu)\in\RandPO(\mu)$) with $\zeta \And \nu \preceq_\mu \nu$. Therefore $D_\gamma$ is dense, and so there must be some $\beta \in \Generic \cap D_\gamma$. However, as $\gamma \in \Generic$, there must be some $\zeta \in \Gamma$ such that $\beta \preceq_\mu \zeta$. Therefore,  there must be some $\zeta \in \Gamma \cap \Generic$. By induction, we therefore have $V[\Generic] \models \cM_G \in \BCvalue(\zeta) \subseteq \BCvalue(\gamma)$. 

The proof of \eqref{mu-generic filters are determined by their intersection: Eq 1} when $\gamma = \bigwedge \Gamma$ is similar. 
Therefore, by induction, \linebreak \eqref{mu-generic filters are determined by their intersection: Eq 1} holds and so (c) holds as well. 
\end{proof}

\begin{definition}
Suppose $\mu$ is a probability $\kappa$-measure on $\Str<\kappa>[\Lang](X)$ and $\Generic$ is a $\mu$-random generic filter over $V$. Let $\cM_{\Generic}$ be the $\Lang$-structure such that \linebreak $V[\Generic] \models \bigcap_{g \in \Generic}\BCvalue(g) = \{\cM_{\Generic}\}$. We call $\cM_{\Generic}$ a \defn{$\mu$-random generic structure}. 
\end{definition}

The following is a key property of $\mu$-random generic structures. 

\begin{proposition}
\label{Random generic model satisfies theory of the measure}
Suppose that $X, \Lang \in V$ and that  $V\models |X| \leq \kappa$.
Let $\mu \in V$ be a probability $\kappa$-measure on $[\Str<\kappa>[\Lang](X)]^V$ and let $\Generic$ be a $\mu$-random generic filter over $V$. Then $\cM_{\Generic} \models [\Th[\kappa](\mu)]^{V}$. 
\end{proposition}
\begin{proof}
By \cref{Exists Borel codes for formulas}, for every sentence $\varphi \in [\Lkpw(\Lang)]^{V}$ 
there is a \linebreak $\BCformula(\varphi) \in [\BorelCodes<\kappa>[\Lang](X)]^{V}$  such that $V[\Generic] \models \BCvalue(\BCformula(\varphi)) = \extent<\Lang>[X]{\varphi}$. But for all $\varphi \in [\Th[\kappa](\mu)]^{V}$ 
we have $V \models \mu(\BCformula(\varphi)) = 1$, and so  $\BCformula(\varphi) \in \Generic$. But then by \cref{mu-generic filters are determined by their intersection} 
we have $V[\Generic] \models \cM_{\Generic} \in \BCvalue(\BCformula(\varphi))$, and so $\cM_{\Generic} \models \varphi$. 
\end{proof}

An ergodic invariant measure on $\Str<\w>[\Lang](X)$ is said to be \defn{properly} \linebreak \defn{ergodic} when it assigns measure~$0$ to the isomorphism class of each structure in $\Str<\w>[\Lang](X)$.
In \cite{AFKrP}
it was shown that for every properly ergodic invariant measure $\mu \in V$, there is no model in $V$ of $[\Th[\kappa](\mu)]^V$.
However, \cref{Random generic model satisfies theory of the measure} establishes that
for any $\mu$-random generic filter $\Generic$, there is such a model in $V[\Generic]$,
even though $V$ and $V[\Generic]$ have the same cardinals.

In the case of 
non-properly 
ergodic invariant measures,
we have the following consequence of
\cref{Random generic model satisfies theory of the measure}.
Recall that two $\Lang$-structures are \emph{potentially isomorphic} when they satisfy the same sentences of $\Liw(\Lang)$.

\begin{proposition}
\label{Generics for invariant measure concentrated on a structure give rise to potentially isomorphic structures}
Let $\Lang$ be a countable language, and
let $\mu$ be an invariant 
measure on $\Str<\kappa>[\Lang](\w)$
that concentrates on the isomorphism class of some structure.
Suppose $\Generic$ is a $\mu_X$-random generic filter over $V$ and 
suppose $\Generic[H]$ is a $\mu_Y$-random generic filter over $V$.
Then $\cM_{\Generic}$ and $\cM_{\Generic[H]}$ are potentially isomorphic. 
\end{proposition}
\begin{proof}
Suppose $\mu$ is concentrated on the isomorphism class of $\cN$. Because $\Lang$ is countable there is a sentence $\varphi_{\cN} \in \Lwow(\Lang)$ such that for any \linebreak $\cN^* \in \Str[\Lang](\w)$, we have $\cN^* \models \varphi$ if and only if $\cN^* \cong \cN$. In particular, $\mu$ concentrates on $\BCvalue(\BCformula(\varphi_{\cN}))$ and hence $\varphi_{\cN} \in \Th[\kappa](\mu)$. 

Therefore, by \cref{Theory of an invariant is absolute}, we have that $\varphi_{\cN} \in \Th[\kappa](\mu_X)$ and $\varphi_{\cN} \in \Th[\kappa](\mu_Y)$. So,  by \cref{Random generic model satisfies theory of the measure}, $\cM_{\Generic} \models \varphi_{\cN}$ and $\cM_{\Generic[H]} \models \varphi_{\cN}$. Hence $\cM_{\Generic}$ and $\cM_{\Generic[H]}$ are potentially isomorphic.
\end{proof}

For a countable language $\Lang$,
suppose $\mu$ is a properly ergodic invariant probability measure
on $\Str<\w>[\Lang](\w)$. If $\Generic$ is a $\mu_X$-random generic and $\Generic[H]$ is a $\mu_Y$-random generic,
 then $\cM_{\Generic}$ need not be potentially isomorphic to $\cM_{\Generic[H]}$, as the following example shows. 
\begin{example}
\label{Example that two generics give rise to different structures}
Suppose $\Lang = \{U_i\}_{i \in \w}$ where $U_i$ is a unary relation symbol for each $i \in \w$. Let $\mu$ be the measure where $\extent[\w]{U_i(n)} = \frac{1}{2}$ for all $i, n \in \w$ and where all such events are independent. 

Let $\Generic$ be generic for the measure $\mu \times \mu$ on $\Str<\kappa>[\Lang](\w) \times \Str<\kappa>[\Lang](\w)$.  Let $\Generic_0 = \{X \st X \times \Str<\kappa>[\Lang](\w) \in \Generic\}$ and $\Generic_1 = \{X \st \Str<\kappa>[\Lang](\w) \times X \in \Generic\}$. Note that both $\Generic_0$ and $\Generic_1$ are generic for $\mu$. 

Consider the event 
\[
E_{n_0, n_1} = \Bigl\{\cN_0 \times \cN_1 \in \Str[\Lang](\w) \times \Str[\Lang](\w) \st \bigwedge_{i \in \w} \bigl[\cM_0 \models U_i(n_0) \leftrightarrow \cM_1 \models U_i(n_1)\bigr]\Bigr\}.
\]
Note that $(\mu \times \mu)(E_{n_0, n_1}) = 0$ for all $n_0, n_1 \in \w$. Therefore we must have $\cM_{\Generic_0} \not \cong \cM_{\Generic_1}$. 
\end{example}

The ideas in \cref{Example that two generics give rise to different structures} can be combined with the results of \cite{AFKrP} to show that for any properly ergodic 
invariant probability
measure $\mu$ on $\Str<\w>[\Lang](\w)$, whenever $\Generic$ is generic for $\mu \times \mu$ and $\Generic_0, \Generic_1$ are as in
\cref{Example that two generics give rise to different structures},  
we have
$\cM_{\Generic_0} \not \cong \cM_{\Generic_1}$.

The following is a key property 
of $\mu$-random forcings. 

\begin{proposition}
\label{Random forcings on structure are ccc}
Let $\mu$ be a probability $\kappa$-measure
on $\Str<\kappa>[\Lang](X)$.
Then $\RandPO(\mu)$ has c.c.c.
\end{proposition}
\begin{proof}
Suppose $\langle \gamma_\alpha: \alpha < \omega_1  \rangle \subseteq \RandPO(\mu)$ is a sequence of pairwise incompatible elements. In particular, for each $\alpha \neq \beta$ we have $\mu(\gamma_\alpha \And   \gamma_\beta)=0$.
For some $n<\omega$, the set
$\{ \alpha < \omega_1 : \mu(\gamma_\alpha) > \frac{1}{n} \}$ is uncountable. Let $\alpha_0 < \cdots < \alpha_n < \omega_1$ be such that $\mu(\gamma_{\alpha_i}) > \frac{1}{n}$, for all $i \leq n$. Then
\[
\mu\bigg(\bigvee_{i \in [n+1]}\gamma_{\alpha_i}\bigg) = \sum_{i \in [n+1]}\mu(\gamma_{\alpha_i}) > \frac{n+1}{n} > 1,
\]
a contradiction.
\end{proof}

The following tells us that from the perspective of asking which generics exist, forcing with any non-atomic invariant probability measure 
is as good as 
forcing with any other,
provided the languages and sets are countably infinite.

\begin{proposition}
\label{Any two invariant measures on countable sets force to give the same universes}
For $i \in \{0, 1\}$, suppose that
$\mu_i$ is an invariant
probability
measure on $\Str<\w>[\Lang_i](X_i)$ which assigns measure $0$ to each trivial structure.
Let $V$ be a model of set theory such that 
$X_i, \Lang_i, \mu_i \in V$ and \[
V\models |X_i| =\w    \qquad \text{and} \qquad V \models |\Lang_i| \le \w
\]
for each $i \in \{0, 1\}$.
Let $\Generic_0$ be a $\mu_0$-random generic.
Then there is a $\Generic_1 \in V[\Generic_0]$ which is $\mu_1$-random generic over $V$. 
\end{proposition}
\begin{proof}
Because
$V \models \max\{|X_i|, |\Lang_i|\} = \w$ 
for $i \in \{0, 1\}$,  
by Lemmas~\ref{Isomorphism between Str_L and Cantor} and  
\ref{Isomorphism between Borel codes on Str_L and Cantor} there must be an isomorphism between $\Str<\w>[\Lang_i](X_i)$ and $\Cantor[\w]$ which lifts to an isomorphism of Borel codes. Further, as $\mu_i$ is invariant and assigns measure $0$ to each trivial structure, $\mu_i$ has no point masses, i.e., is non-atomic. Suppose $\nu_i$ is the measure on $\Cantor[\w]$ which corresponds
to $\mu_i$. Then $\nu_i$ is non-atomic as well. Hence by Maharam's theorem there is a measure-preserving bijection $f$ from $\Cantor[\w]$ to itself such that $\nu_0 = \Pushforward[f](\nu_1)$. The result follows immediately. 
\end{proof}

\section{Distinguishing Generic Structures}
\label{Section: Distinguishing Generic Structures}

In this section we consider, for $\mu$ an invariant measure on $\Str<\kappa>[\Lang](X)$, how much information about $\mu$ we can recover from a single $\mu$-random generic filter. In particular, we will be interested in 
determining
when we can uniquely identify $\mu$ from any $\mu$-random generic filter.

Unfortunately,
from a single $\mu$-random generic filter,
we cannot completely identify $\mu$ 
among the invariant measures on $\Str<\kappa>[\Lang](X)$.
This is because for any invariant measure $\mu^*$,
every $\mu$-random generic filter is also a 
$\frac{1}{2}(\mu + \mu^*)$-random generic filter, 
as the following result shows. 

\begin{lemma}
\label{Generic filters are generic for mixtures}
Let $\mu$ be a probability $\kappa$-measure on $\Str<\kappa>[\Lang](X)$.
Suppose \linebreak $X, \Lang, \mu \in V$.
Let $\Generic$ be a $\mu$-random generic filter and let $\gamma \in \Generic$.
For $\zeta \in [\BorelCodes<\kappa>[\Lang](X)]^V$ define $\nu(\BCvalue(\zeta)) = \frac{\mu(\zeta \And \gamma)}{\mu(\gamma)}$. Then the following hold. 
\begin{itemize}
\item[(a)] $\nu$ is a probability $\kappa$-measure 
on $(\Str<\kappa>[\Lang](X))^V$.

\item[(b)] 
If $\mu$ is invariant and $\gamma$ is $\mu$-almost invariant, then $\nu$ is invariant.
\item[(c)] $\Generic$ is a $\nu$-random generic filter over $V$. 
\end{itemize}

\end{lemma}
\begin{proof}
To prove (a), note that as $\gamma \in \Generic$ we have $\mu(\gamma) > 0$.

To prove (b), suppose $\gamma$ is $\mu$-almost invariant and $g \in [\Sym(X)]^V$.
Then 
\[
\nu(g[\zeta]) = \frac{\mu(g[\zeta] \And \gamma)}{\mu(\gamma)} =  \frac{\mu(g[\zeta] \And g[\gamma])}{\mu(g[\gamma])} = \frac{\mu(\zeta \And \gamma)}{\mu(\gamma)}  = \nu(\zeta).
\]

Finally, to prove (c), suppose $D$ is dense in $\RandPO(\nu)$. Let \linebreak $D^* = D \cup \{\zeta \preceq_\mu \neg \gamma\}$. Suppose $\zeta \in \RandPO(\mu)$, and hence $\mu(\zeta) > 0$. Then either $\mu(\zeta \And \gamma) > 0$ or $\mu(\zeta \And \neg \gamma) > 0$. But if $\mu(\zeta \And \gamma) > 0$ then $\zeta \in \RandPO(\nu)$.
Therefore $D^*$ is dense in $\RandPO(\mu)$. 

Let $\zeta \in \RandPO(\nu)$. Then $\nu(\zeta) > 0$ and so $\mu(\zeta \And \gamma) > 0$. Therefore \linebreak $\zeta \And \gamma \in \RandPO(\mu)$ and so there must be an $\eta \in D^* \cap \Generic$ with $\eta \preceq_{\mu} \zeta \And \gamma$. But then $\eta \in D$. Hence $\Generic$ is a $\nu$-random generic filter. 
\end{proof}

\cref{Generic filters are generic for mixtures} (b) and (c) suggest that if we want to be able to identify an invariant measure from a single random generic filter for the measure, we should restrict attention to measures $\mu$ that assign measure $0$ or $1$ to
every $\mu$-almost invariant set,
i.e., which are ergodic.

\begin{definition}
Let $\varphi$ be a first-order non-redundant quantifier-free $\Lang$-formula. For $\cN$ a finite $\Lang$-structure, 
define the \defn{density of $\varphi$ in $\cN$} by 
\[
t(\cN, \varphi) = \frac{|\{\xx \in \NonRedTuple[\arity(\varphi)](N) \st \cN \models \varphi(\xx)\}|}{|\NonRedTuple[\arity(\varphi)](N)|}.
\]
For $I \subseteq [0, 1]$ and $i\:\w \to X$ let 
\[L(i, \varphi, I) = \bigl\{\cM \in \Str<\kappa>[\Lang](X) \st \lim_{n \to \infty} t(\cM\rest[{i``[n]}], \varphi) \in I\bigr\}.
\]
In particular, $L(i, \varphi, [0, 1]) = \{\cM \in \Str<\kappa>[\Lang](X) \st \lim_{n \to \infty} t(\cM\rest[{i``[n]}], \varphi)\text{ exists}\}$.
\end{definition}

We will need the following.

\begin{lemma}
\label{Borel codes for limiting probabilities}
For every first-order non-redundant quantifier-free $\Lang$-formula $\varphi$, every open $I \subseteq [0, 1]$, and every injective $i\:\w \to X$, there is a \linebreak $\zeta(i, \varphi, I) \in \BorelCodes<\kappa>[\Lang](X)$ of rank 4 such that $\BCvalue(\zeta(i, \varphi, I)) = L(i, \varphi, I)$ in any model of set theory containing  $X$ and $\Lang$ and a code for $I$. 
\end{lemma}
\begin{proof}
For $J \subseteq [0, 1]$ let $K(n, i, \varphi, J)\in \AlmostBasicBC<\kappa>[\Lang](X)$ be such that
\[
\BCvalue(K(n, i, \varphi, J)) = \{\cM \in \Str<\kappa>[\Lang](X) \st  t(\cM\rest[{i``[n]}], \varphi) \in J\}.
\]
Let $\alpha(i, \varphi, J) = \bigvee_{k \in \w} \bigwedge_{ n \geq k }K(n, i, \varphi, J)$. 
Then $\BCvalue(\alpha(i, \varphi, J))$ consists of those structures for which the density of $\varphi$ in increasingly large restrictions
is eventually in $J$.

For $k \in \w$, let $\beta_k(i, \varphi) = \bigvee_{0 \leq \ell < 2^k}\alpha(i, \varphi,  [\frac{\ell}{2^k}, \frac{\ell+1}{2^k}])$. 
Then 
$\BCvalue(\beta_k(i, \varphi))$ is the collection of $\cM$ 
n $\Str<\kappa>[\Lang](X)$
such that for some interval of length $2^{-k}$, the values of $t(\cM\rest[{i``[n]}], \varphi)$ are always eventually in that interval. 
Because $I$ is open, we can then let $\zeta(i, \varphi, I) = \alpha(i, \varphi, I) \And \bigwedge_{k \in \w} \beta_k(i, \varphi)$. 
\end{proof}    

The following is a useful property about sampling from invariant measures. 

\begin{proposition}
\label{Sampling finite structures approximates a measure}
Let $\mu$ be an invariant measure on $\Str<\kappa>[\Lang](X)$.
Let $i\:\w \to X$ be an injection, 
and let $\varphi(\xx)$ be a first-order quantifier-free $\Lang$-formula.
Suppose $\cM$ is a $\mu$-random $\Lang$-structure.
Then for every $\epsilon > 0$, there is some value $k(\epsilon, \varphi)$ 
such that 
\[
\mu \bigl(\bigl\{ \cM \st (\forall m > k(\epsilon, \varphi))\, \bigl|t(\cM \rest[{i[m]}], \varphi) - \mu(\BCformula(\varphi(i(0), \dots, i(|\xx|-1))))\bigr| < \epsilon \bigr\}\bigr) > 1 - \epsilon.
\] 
\end{proposition}
\begin{proof}
This follows from a union bound applied to a straightforward generalization to finite relational structures of the concentration result 
\cite[Theorem~10.3]{MR3012035}.
\end{proof}   

Thus we have the following result. 

\begin{proposition}
\label{Subgraph densities define a measure}
Let $\mu$ be an invariant measure on $\Str<\kappa>[\Lang](X)$.
Let $i\:\w \to X$ be an injection, 
and let $\varphi(\xx)$ be a first-order quantifier-free $\Lang$-formula.
If $\mu$ is ergodic,
then the set
\[
L\bigl(i, \varphi, \bigl\{\mu(\BCformula(\varphi(i(0), \dots, i(\arity(\varphi)- 1)))\bigr\}\bigr)
\]
has $\mu$-measure $1$.
\end{proposition}
\begin{proof}
Let $\epsilon > 0$ and let \[
I_\epsilon = \Bigl(\mu\bigl(\BCformula(\varphi(i(0), \dots, i(\arity(\varphi)- 1)))\bigr) - \epsilon,\ \mu\bigl(\BCformula(\varphi(i(0), \dots, i(\arity(\varphi)- 1)))\bigr) + \epsilon\Bigr). 
\]

For $n \in \w$, let $\ell_n = k(\epsilon \cdot 2^{-n-1}, \varphi)$ where $k$ is as in \cref{Sampling finite structures approximates a measure}. Let 
\[
J_{n, \epsilon} =\bigl\{\cM \st (\exists m > \ell_n)\,  
\bigl|t(\cM \rest[{i[m]}], \varphi) - \mu(\BCformula(\varphi(i(0), \dots, i(\arity(\varphi)- 1)))\bigr| \geq \epsilon \cdot 2^{-n-1}\bigr\}.
\]
Then $\Str<\kappa>[\Lang](X) \setminus \bigcup_{n \in \w} J_{n, \epsilon} \subseteq L(i, \varphi, I_\epsilon)$.

But by \cref{Sampling finite structures approximates a measure} we have that $\mu(J_n) \leq \epsilon \cdot 2^{-n-1}$. Therefore \linebreak $\mu(L(i, \varphi, I_\epsilon)) \geq 1 - \epsilon$. As $\epsilon$ was arbitrary and 
\[\bigcap_{\epsilon > 0} I_\epsilon = \bigl\{\mu(\BCformula(\varphi(i(0), \dots, i(\arity(\varphi)- 1)))\bigr\},
\]
we have $\mu\Bigl(L\bigl(i, \varphi, \{\mu(\BCformula(\varphi(i(0), \dots, i(\arity(\varphi)- 1)))\})\bigr)\Bigr) = 1$. 
\end{proof}

In particular, this implies that if $\mu_0$ and $\mu_1$ are distinct ergodic invariant measures
and $\Generic$ is either a generic for $\mu_0$ or a generic for $\mu_1$, we can determine 
the measure for which $\Generic$ is a generic
by looking at $\cM_{\Generic}$.

\begin{corollary}
\label{There are Borel codes isolating ergodic invariant measures}
Let $\mu_0$ and $\mu_1$ be distinct ergodic invariant measures
on \linebreak $\Str<\kappa>[\Lang](X)$. Then there is some $U \in \BorelCodes<\kappa>[\Lang](X)$ of rank 5 such that $\mu_0(U) = 1$ and $\mu_1(U) = 0$. 
\end{corollary}
\begin{proof}
Let
$i\:\w \to X$ be an injection. 
As $\mu_0$ and $\mu_1$ are distinct invariant measures
there must be a first-order quantifier-free $\Lang$-formula $\psi$ 
such that 
\[\mu_0(\BCformula(\psi(i(0), \dots, i(n-1)))) \neq \mu_1(\BCformula(\psi(i(0), \dots, i(n-1)))),
\]
where $n$ is the arity of $\psi$.
The result then follows from \cref{Borel codes for limiting probabilities} and \cref{Subgraph densities define a measure}.
\end{proof}

This immediately implies the following corollary. 

\begin{corollary}
Let $\mu_0$ and $\mu_1$ be distinct ergodic invariant measures
on \linebreak $\Str<\kappa>[\Lang](X)$.
If $\Generic$ is a $\mu_0$-random generic filter, then $\Generic$ is not a $\mu_1$-random generic filter.  
\end{corollary}

In general, when $\mu$ is not ergodic, $L(\varphi, \{\mu(\extent[\w]{\varphi})\})$ need not have \linebreak $\mu$-measure $1$.
 In other words, there might not be a single limiting probability to which almost all samples converge.
However, any invariant measure will concentrate on the collection of structures for which there is almost always some limiting value.

\begin{proposition}
\label{Limiting densities exist for random generics}
Let $\mu$ be an invariant measure
on $\Str<\kappa>[\Lang](X)$.
Then \linebreak $\mu(\zeta(i, \varphi, [0, 1])) = 1$ for any injection $i\:\w \to X$ and first-order quantifier-free $\Lang$-formula $\varphi$.
\end{proposition}
\begin{proof}
By \cref{Relativizations give same values} and \cref{Borel codes for limiting probabilities}, and by moving to a larger model of set theory if needed, we can assume that $|X| = |\Lang| = |\kappa| = \w$. 

By \cite[Theorem~A1.3]{MR2161313}, the measure $\mu$ is a mixture of ergodic invariant 
measures. But by \cref{Subgraph densities define a measure}, every ergodic invariant 
measure assigns measure $1$ to $\zeta(i, \varphi, [0, 1])$. Therefore $\mu$ is a mixture of measures for which $\zeta(i, \varphi, [0, 1])$ has measure $1$, and hence $\mu(\zeta(i, \varphi, [0, 1])) = 1$.  
\end{proof}

\begin{definition}
A structure $\cM$ is \defn{highly homogeneous} if whenever $A, B \subseteq \cM$ with $|A| = |B| < \w$, there is some $g \in\Aut(\cM)$ such that $g``[A] = B$. 
\linebreak A strong \Fraisse\ class is \defn{highly homogeneous} when any two structures in 
the class
of the same size
are isomorphic.
\end{definition}
Note that a strong \Fraisse\ class is highly homogeneous precisely when it has a highly homogeneous \Fraisse\ limit.

Up to interdefinability, there are only five countably infinite highly homogeneous structures, namely the reducts of $(\Rationals, \leq)$, as shown in  \cite{MR401885}. Highly homogeneous structures are important for us because of the following result. 

\begin{proposition}[\hbox{\cite{MR3564374}}]
\label{Unique invariant measure}
The following are equivalent for a countably infinite structure $\cM$ in a countable language $\Lang$. 
\begin{itemize}
\item $\cM$ is highly homogeneous.

\item There is a unique invariant 
probability
measure concentrated on the isomorphism class of $\cM$ in $\Str<\w>[\Lang](\w)$, and further this invariant probability measure is ergodic. 
\end{itemize}
\end{proposition}

Given a generic filter that is either a Cohen generic filter for a strong \Fraisse\ class that is not highly homogeneous or 
that 
is a $\mu$-random generic filter,
\cref{Borel codes separate random generic from Cohen generic}
will allow us to 
determine which of these is the case. This will be a consequence of the following result.

\begin{theorem}
\label{Limiting densities don't exist for Cohen generics}
Let $\AgeK$ be a strong \Fraisse\ class of $\Lang$-structures
that is not highly homogeneous.
Let $\cM_G$ be the result of forcing with $\Cohen[X](\AgeK)$, i.e., Cohen forcing with $\AgeK$
on $X$. Suppose $i\:\w \to X$ is an injection. Then there is a first-order quantifier-free 
$\Lang$-formula
$\varphi$ such that 
\[
\limsup_{n \to \infty} \, t(\cM_G\rest[{i``[n]}], \varphi)\neq 
\liminf_{n \to \infty} \, t(\cM_G\rest[{i``[n]}], \varphi). 
\]
In particular, we have 
\[
\cM_G \not \in \bigcap \{\zeta(i, \psi, [0, 1])\st \psi\in \Lww(\Lang)\text{ is quantifier-free}\} .
\]
\end{theorem}
\begin{proof}
First we prove two claims. 

\begin{claim}
\label{Limiting densities don't exist for Cohen generics:Claim 1}
There is
some 
first-order
quantifier-free $\Lang$-formula $\varphi$
of arity $k > 0$ and there are
rationals $q_0$ and $q_1$ 
with
$0 < q_0 < q_1 <  1$
such that for all $n \in \w$, there exist $\cA_0^n, \cA_1^n \in \AgeK$ with
\begin{itemize}
\item $|\cA_0^n| = |\cA_1^n| > n$, and

\item $t(\cA_0^n, \varphi)< q_0 < q_1 < t(\cA_1^n, \varphi)$.
\end{itemize}
\end{claim}
\begin{proof}
First note that whether or not this claim holds is independent of the model of set theory. Therefore, by moving to a forcing extension if needed, we can assume without loss of generality that $|\Lang| = |X| = \w$ and that $\AgeK$ has only countably many isomorphism types. 

Let $\cN$ be a countably infinite \Fraisse\ limit of $\AgeK$. Because $\AgeK$ has two non-isomorphic elements of the same finite size, $\cN$ is not highly homogeneous. Therefore by \cref{Unique invariant measure} there are two distinct ergodic invariant 
probability
measures $\mu_0, \mu_1$ on $\Str<\w>[\Lang](i``(\w))$ concentrated on the isomorphism class of $\cN$.
Hence there must be some first-order quantifier-free 
$\Lang$-formula
$\varphi$ of arity $k \in \w$ and rationals $q_0, q_1$ such that 
\[
\mu_0\bigl(\BCformula(\varphi(i(0), \dots, i(k-1)))\bigr) \ <\  q_0 \ <\  q_1  \ < \ 
\mu_1\bigl(\BCformula(\varphi(i(0), \dots, i(k-1)))\bigr).
\]
Note that because $\mu_0$ and $\mu_1$ are both concentrated on the isomorphism class of $\cN$, the formula $\varphi$ cannot be a sentence.  The claim then follows from \cref{Sampling finite structures approximates a measure}.
\end{proof}

Now fix $q_0, q_1$, $k$, $\varphi$, and $A_0^n$, $A_1^n$ for $n \in \w$, as in \cref{Limiting densities don't exist for Cohen generics:Claim 1}.

\begin{claim}
\label{Limiting densities don't exist for Cohen generics:Claim 2}
For every $\cB \in \AgeK$  there exist $\cC_0, \cC_1 \in \AgeK$ satisfying
\begin{itemize}
\item $\cB \subseteq \cC_0$ and $\cB \subseteq \cC_1$, 

\item $|C_0| = |C_1|$,  and

\item $t(\cC_0, \varphi) \leq q_0 < q_1 \leq t(\cC_1, \varphi)$.
\end{itemize}
\end{claim}
\begin{proof}
Let $\epsilon > 0$.
Let $n_\epsilon$ be such that 
\[
\frac{|\NonRedTuple[k]([n_\epsilon])|}{|\NonRedTuple[k]([n_\epsilon + |B|])|} > 1 - \frac{\epsilon}{2}
\]
and
\[
\frac{|B| \cdot |\NonRedTuple[k-1]([n_\epsilon + |B|])|}{|\NonRedTuple[k]([n_\epsilon + |B|])|} < \frac{\epsilon}{2}.
\]
Recall that $\AgeK$ is a strong \Fraisse\ class, and hence has the joint embedding property. For $j \in \{0, 1\}$ let $\cC_j$ be any structure of size $|B| + |A_j^{n_\epsilon}|$ such that both $\cB$ and $\cA_j^{n_\epsilon}$ have non-overlapping 
embeddings into $\cC_j$.
We then have 
\begin{align*}
\frac{t(\cA_j^{n_\epsilon}, \varphi) \cdot |\NonRedTuple[k]({(A_j^{n_\epsilon})})|}{|\NonRedTuple[k](C_j)|} 
\, &\leq\, t(\cC_j, \varphi) \\
&\leq\,  \frac{t(\cA_j^{n_\epsilon}, \varphi) \cdot |\NonRedTuple[k]({(A_j^{n_\epsilon})})| + |B| \cdot |\NonRedTuple[k-1]([|B| + |A_j^{n_\epsilon}|])|}{|\NonRedTuple[k](C_j)|}
\end{align*}
where the left inequality follows by assuming no tuple in $\cC_j$ with an element from $B$ satisfies $\varphi$ and the right inequality follows by assuming all tuples in $\cC_j$ which have an element from $B$ satisfy $\varphi$. 

We therefore have 
\begin{align*}
q_1  - \frac{\epsilon}{2} 
\,&<\, t(\cA_1^{n_\epsilon}, \varphi) - \tfrac{\epsilon}{2} \\
\,&\leq\, t(\cA_1^{n_\epsilon}, \varphi) \cdot (1 - \tfrac{\epsilon}{2})\\
\,&<\, \frac{t(\cA_1^{n_\epsilon}, \varphi) \cdot |\NonRedTuple[k]({(A_1^{n_\epsilon})})|}{|\NonRedTuple[k](C_1)|} \\
\,&\leq\, t(\cC_1, \varphi).
\end{align*}
Because $\epsilon> 0$ was arbitrary, this implies that $q_1 \leq t(\cC_1, \varphi)$. 
We also have 
\begin{align*}
t(\cC_0, \varphi)
\,&\leq\, \frac{t(\cA_0^{n_\epsilon}, \varphi) \cdot |\NonRedTuple[k]({(A_0^{n_\epsilon})})|}{|\NonRedTuple[k](C_0)|} + \frac{ |\NonRedTuple[k-1]([|B| + |A_0^{n_\epsilon}|])|}{|\NonRedTuple[k](C_0)|} \\
\,&<\,t(\cA_0^{n_\epsilon}, \varphi) + \tfrac{\epsilon}{2} \\
\,&<\, q_0 + \tfrac{\epsilon}{2}.
\end{align*}
Hence, as $\epsilon$ was arbitrary, we have $t(\cC_0, \varphi) \leq q_0$, proving the claim.
\end{proof}

Now let $D_0 = \{\cE \in \AgeK \st t(\cE, \varphi) \leq q_0\}$. By \cref{Limiting densities don't exist for Cohen generics:Claim 2}, $D_0$ is dense, and so we have 
$
\liminf_{n \to \infty}(\cM_G\rest[{i``[n]}], \varphi) \,\leq\, q_0
$.

Let $D_1 = \{\cE \in \AgeK \st t(\cE, \varphi) \geq q_1\}$.  By \cref{Limiting densities don't exist for Cohen generics:Claim 2}, we also have that  $D_1$ is dense. Hence $\limsup_{n \to \infty}(\cM_G\rest[{i``[n]}], \varphi) \geq q_1$. 

Therefore 
\[
\liminf_{n \to \infty}\, (\cM_G\rest[{i``[n]}], \varphi) \,\leq\, q_0 \,<\, q_1 \,\leq\, \limsup_{n \to \infty}\,(\cM_G\rest[{i``[n]}], \varphi),
\]
as desired, establishing the theorem.
\end{proof}

\begin{theorem}     
\label{Borel codes separate random generic from Cohen generic}
There is a $U \in \BorelCodes<\kappa>[\Lang](X)$ with $\rank(U) = 5$ such that if $\cM$ is Cohen generic on $X$ for some strong \Fraisse\ class of $\Lang$-structures
that is not highly homogeneous then $\cM \not \in \BCvalue(U)$, and if $\cM$ is $\nu$-random generic for some invariant measure $\nu$ 
on $\Str<\kappa>[\Lang](X)$
then $\cM \in \BCvalue(U)$. 
\end{theorem} 
\begin{proof}  
Let 
$
U = \bigwedge \{\zeta(i, \varphi, [0, 1])\st \varphi \in \Lww(\Lang)\text{ is quantifier-free}\}$.
The result then follows from \cref{Limiting densities exist for random generics} and \cref{Limiting densities don't exist for Cohen generics}. 
\end{proof}

The following is an immediate corollary of \cref{Borel codes separate random generic from Cohen generic}.  
\begin{corollary}
Suppose $\cM$ is Cohen generic on $X$ for some strong \Fraisse\ class of $\Lang$-structures
that is not highly homogeneous. Then $\cM$ is not $\nu$-random generic for any invariant 
measure
$\nu$ on $\Str<\kappa>[\Lang](X)$. 

Similarly if $\cM$ is $\nu$-random generic for some invariant 
measure
$\nu$ on $\Str<\kappa>[\Lang](X)$ then $\cM$ is not Cohen generic on $X$ for any strong \Fraisse\ class of $\Lang$-structures that is not highly homogeneous. 
\end{corollary}

\section{Conjectures}
\label{Section: Conjectures}

We end with some conjectures about $\mu$-random generics and Cohen generics. 

As we saw in \cref{All Fraisse classes give the same universe}, whenever $\AgeK_0$ and $\AgeK_1$ are non-trivial countable strong \Fraisse\ classes in countable languages, if $\Generic_0$ is generic for $\AgeK_0$ on $X$ over $V$ then $V[\Generic_0]$ contains a filter $\Generic_1$ which is generic for $\AgeK_1$ on $\w$ over $V$. We conjecture that countability is not needed. 
\begin{conjecture}
Suppose 
$\AgeK_0$ and $\AgeK_1$ are non-trivial \Fraisse\ classes each with at most $|X|$-many elements up to isomorphism. Then $\Cohen[X](\AgeK_0)$ and $\Cohen[X](\AgeK_1)$ are forcing equivalent. 
\end{conjecture}

We also saw in \cref{Any two invariant measures on countable sets force to give the same universes} that when the set and language are countably infinite, any two invariant probability measures that assign measure $0$ to every trivial structure give rise to equivalent notions of forcing. This suggests the following conjecture.

\begin{conjecture}
Let $\mu_0, \mu_1 \in V$ be invariant measures on $\Str<\kappa>[\Lang](X)$, and let $\Generic_0$ be a $\mu_0$-random generic filter over $V$. 
Suppose that $\mu_0$ and $\mu_1$ each assign measure $0$ to every trivial structure.
Then in $V[\Generic_0]$ there is a $\mu_1$-random generic filter over $V$. 
\end{conjecture}

We also make a stronger conjecture about the nature of invariant measures, motivated by the proof of \cref{Any two invariant measures on countable sets force to give the same universes}.
\begin{conjecture}
Let $\mu_0$ and $\mu_1$ be invariant measures on $\Str<\kappa>[\Lang](X)$. 
Suppose that $\mu_0$ and $\mu_1$ each assign measure $0$ to every trivial structure.
Then there is a measurable isomorphism $f \:\Str<\kappa>[\Lang](X) \to \Str<\kappa>[\Lang](X)$ such that $\mu_0$ is the pushforward of $\mu_1$ along $f$, i.e., $\mu_0 = \Pushforward[f](\mu_1)$.
\end{conjecture}

An important theorem relating random forcing and Cohen forcing is that adding a Cohen real will not add any random reals, and adding a random real will not add any Cohen reals (see \cite[Lemma~3.3]{Jech_1987}). We conjecture that this extends to arbitrary strong \Fraisse\ classes.

\begin{conjecture}
Let $\AgeK$ be a strong \Fraisse\ class in $\Lang$, and let
$\mu$ be an invariant measure on $\Str<\kappa>[\Lang](X)$  which concentrates on the collection of models potentially isomorphic to a \Fraisse\ limit of $\AgeK$.
\begin{enumerate}
\item 
Suppose $\Generic_0$ is a $\mu$-random generic filter over $V$. Then
$V[\Generic_0]$ contains no Cohen generic structures for $\AgeK$ on $X$ over $V$.

\item 
Suppose $\Generic_1$ is a $\AgeK$-generic filter for $X$ over $V$. Then
$V[\Generic_1]$ contains no $\mu$-random generic structures  over $V$. 
\end{enumerate}
\end{conjecture}

An interesting property of Cohen generic structures is that even when the \Fraisse\ limit of the corresponding age has a large automorphism group, the Cohen generic structures are rigid (see \cite{Golshani} and \cite{Cohen} for examples and \cite{Cohen-Generic-with-Functions_AGM} for a property which ensures rigidity). This suggests that a similar phenomenon might occur for $\mu$-random generic structures.

\begin{definition} Let $n\ge 1$.
A \Fraisse\ class 
$\AgeK$ has the \defn{disjoint $n$-amalgamation property} (abbreviated $n$-DAP) if whenever $\<\cA_i\>_{i \in [n]} \subseteq \AgeK$ are such that 
$\cA_i\rest[A_i \cap A_j] = \cA_j\rest[A_i \cap A_j]$
for all $i < j \in [n]$, 
there is some $\cB \in \AgeK$ such that 
$\cA_i\subseteq \cB$
for all $i \in [n]$.
\end{definition}

Observe that $2$-DAP is the same as the strong amalgamation property.

\begin{conjecture}
Let $\AgeK$ be a strong \Fraisse\ class that satisfies $n$-DAP for every $n$ and has only countably many isomorphism types. Let $\cN$ be a countably infinite \Fraisse\ limit of $\AgeK$. Suppose that $\mu$ is an invariant measure concentrated on the isomorphism class of $\cN$ and that $\Generic$ is a $\mu$-random generic filter. Then $\cM_{\Generic}$ is rigid. 
\end{conjecture}

We showed in 
\cref{Borel codes separate random generic from Cohen generic}
that whenever a strong \Fraisse\ class $\AgeK$ is not highly homogeneous, there is some Borel set that contains all Cohen generic structures but contains no $\mu$-random generic structure for any invariant measure $\mu$. We conjecture that this holds for all non-trivial strong \Fraisse\ classes. 

\begin{conjecture}
For any non-trivial strong \Fraisse\ class $\AgeK$, 
there is a \linebreak $\gamma \in [\BorelCodes<\kappa>[\Lang](X)]^V$ such that 
\begin{itemize}
\item 
$V[\Generic] \models \cM_{\Generic} \in \BCvalue(\gamma)$
whenever $\Generic$ is Cohen generic for $\AgeK$ over $X$, and 

\item 
$V[\Generic] \models \cM_{\Generic} \not \in \BCvalue(\gamma)$ whenever
$\Generic$ is $\mu$-random generic for some invariant measure
$\mu$ that concentrates on the collection of models that are potentially isomorphic to a \Fraisse\ limit of $\AgeK$. 
\end{itemize}
\end{conjecture}

Note that one cannot remove the assumption that $\AgeK$ is non-trivial, as for any trivial strong \Fraisse\ class there is a unique structure with age $\AgeK$ 
for each underlying set.

Recall from  \cref{prop:kal-ergodic} that
when $X$, $\Lang$, and $\kappa$ are countable, every invariant measure on $\Str<\kappa>[\Lang](X)$ is a mixture of ergodic ones. We conjecture that the same is true even when $X$,  $\Lang$, or $\kappa$ is uncountable.

\begin{conjecture}
Suppose $\mu$ is an invariant measure on $\Str<\kappa>[\Lang](X)$. Then $\mu$ is a mixture
of ergodic invariant measures on $\Str<\kappa>[\Lang](X)$. 
\end{conjecture}

\providecommand{\bysame}{\leavevmode\hbox to3em{\hrulefill}\thinspace}
\providecommand{\MR}{\relax\ifhmode\unskip\space\fi MR }
\providecommand{\MRhref}[2]{%
  \href{http://www.ams.org/mathscinet-getitem?mr=#1}{#2}
}
\providecommand{\href}[2]{#2}

\end{document}